\documentclass[a4paper,11pt]{amsart}  
\usepackage{amstext}
\usepackage{amsmath}
\usepackage{ifthen}
\usepackage{amscd}
\usepackage{amssymb}
\usepackage[english]{babel}
\usepackage[T1]{fontenc}
\usepackage[utf8]{inputenc}
\usepackage[all]{xy}
\usepackage{graphicx}
\usepackage{enumerate}
\usepackage{xspace}
\usepackage{epic}
\usepackage{xcolor}

\usepackage[
	hypertexnames=false,
	hyperindex,
	pagebackref,  
	pdftex,
	bookmarks=false,
	colorlinks,
	linkcolor=blue,
	citecolor=red,
	urlcolor=red,
]{hyperref}
\newtheorem{prop}{Proposition}[section]
\newtheorem{coro}[prop]{Corollary}

\newtheorem{lem}[prop]{Lemma}
\newtheorem{rem}[prop]{Remark}

\newtheorem{theoi}{Theorem}

\newtheorem{coroi}[theoi]{Corollary}

\theoremstyle{plain}  
\newtheorem{theorem}[prop]{Theorem}
\newtheorem{lemma}[prop]{Lemma}     
\newtheorem{corollary}[prop]{Corollary}      
\newtheorem{proposition}[prop]{Proposition}

\newtheorem{definition}[prop]{Definition}         
\theoremstyle{remark}      
\newtheorem{remark}[prop]{Remark}

\newcommand{\R}{\mathbb R}
\newcommand{\Q}{\mathbb Q}
\newcommand{\C}{\mathbb C}
\newcommand{\N}{\mathbb N}

\newcommand{\Z}{\mathbb Z}

\newcommand{\Mg}{\mathcal{M}_g}

\newcommand{\Mgnbar}{\overline{\mathcal{M}_{g}^n}}

\newcommand{\Mod}{{\rm{Mod}}} 
\newcommand{\PMod}{{\rm{PMod}}}           

\begin{document}

\title
[Algebro-geometric subgroups]
{Algebro-geometric subgroups of mapping class groups}
\date{\today}

\author[P.~Eyssidieux]{Philippe Eyssidieux}
\address{Univ. Grenoble Alpes, CNRS, Institut Fourier,
38000 Grenoble, France}
\email{philippe.eyssidieux@univ-grenoble-alpes.fr}
\urladdr{http://www-fourier.univ-grenoble-alpes.fr/$\sim$eyssi/}

\author[L.~Funar]{Louis Funar}
\address{Univ. Grenoble Alpes, CNRS, Institut Fourier,
38000 Grenoble, France}
\email{louis.funar@univ-grenoble-alpes.fr} 
\urladdr{http://www-fourier.univ-grenoble-alpes.fr/$\sim$funar/}

\begin{abstract} 
We provide new constraints for algebro-geometric subgroups of mapping class groups, namely images of fundamental groups  of curves under complex algebraic maps to the moduli space of smooth curves. Specifically, we prove that the restriction of an infinite, finite rank unitary representation of the mapping class group of a closed surface to an algebro-geometric subgroup should be infinite, when the genus is at least 3. In particular the restriction of most Reshetikhin-Turaev representations of the mapping class group to such subgroups is infinite. To this purpose we use deep work of Gibney, Keel and Morrison to constrain the Shafarevich morphism associated to a linear representation of the fundamental group of the compactifications of the moduli stack of smooth curves studied in our previous work. As an application we prove that universal covers of most of these compactifications are Stein manifolds.

\vspace{0.2cm}
\noindent 2010 MSC Classification: 14 H 10, 32 Q 30 (Primary) 32 J 25, 32 G 15,   14 D 23,  57 M 07, 20 F 38, 20 G 20, 22 E 40 (Secondary).  
 
\vspace{0.2cm} 
\noindent Keywords: K\"ahler Orbifolds, K\"ahler Groups, Shafarevich morphism,   Mapping class group,
Moduli space of Curves, Deligne-Mumford stacks, Topological Quantum Field Theory.	
\end{abstract}

\maketitle

\tableofcontents

\section{Introduction and statements}
\subsection{Motivation}

Recall that the so-called AMU conjecture \cite{AMU} claims that the image of a pseudo-Anosov mapping class by a Reshetikhin-Turaev representation of large enough level has infinite order.
What we will  show here is a strengthening of a consequence of this conjecture. Namely,  such infiniteness statement holds uniformly for {\em algebro-geometric} subgroups of the mapping class groups, i.e. subgroups
which are images of surface groups, under the condition that they are realized by algebraic maps between the corresponding spaces.
In particular,  this extends results of Koberda and Santharoubane from \cite{KoSa}. The idea of the proof is to use geometric tools such as Shafarevich morphisms \cite{E6} to promote the infiniteness from mapping class groups to  their subgroups. To this purpose we consider the  uniformizable stacky compactifications  of the stack of smooth curves from \cite{mcgorb} and build on classical results about divisors on the moduli space of curves \cite{GKM}. We originally conjectured the infiniteness result above  since it can be applied to show that the universal covering space of these compactifications are Stein manifolds with a few exceptions.

Our results can also be applied to
provide restrictions on the fundamental groups of Kodaira fibrations.
Recall that a {\em Kodaira fibration} is a smooth projective surface $X$ admitting a  holomorphic map
 $\phi: X\to C$ to a smooth algebraic curve $C$,
 which is a smooth fiber bundle such that the fibres are connected. Kodaira and Kas showed that
the genus of the base is then $g\geq 2$ and the genus of fiber is $h\geq 3$ unless the fibration is isotrivial, i.e. isomorphic to a product when pulled-back to a finite \'etale covering of $C$.
 Moreover,  we have a short exact sequence:
\[ 1 \to \pi_1(\Sigma_h) \to \pi_1(X) \to \pi_1(\Sigma_g)\to 1, \]
where $\Sigma_q$ denotes the closed orientable surface of genus $q$.
Viewing $\phi$ as a nonisotrivial family of smooth curves, it provides a nonconstant
holomorphic map $f:C\to {\mathcal M}_g$ into the moduli space  ${\mathcal M}_g$ of smooth genus $g$ curves, whose induced homomorphism $f_*:\pi_1(C)\to \pi_1( {\mathcal M}_g)$ coincides with
the homomorphism $\pi_1(\Sigma_g)\to {\rm Out}^+(\pi_1(\Sigma_h))$ associated to the short exact sequence above.
Note that the outer action $f_*$ determines uniquely the extension and thus a surface
bundle over a surface, up to diffeomorphism.
Catanese (\cite{Ca}, Question 16) asked about conditions needed to be satisfied
by such group extensions in order to occur as the fundamental group of a Kodaira fibration.
The Torelli theorem yields that a Kodaira fibration for which
$f_*(\pi_1(C))$ belongs to the Torelli subgroup of the mapping class group $\Mod(\Sigma_g)$  is isotrivial.  This was considerably strengthened by Arapura in \cite{Ara},  who used Hodge theory to show that restrictions
of certain Prym-type partial representations to algebro-geometric subgroups are infinite  unless $\phi$ is isotrivial. By our results,   $f_*(\pi_1(C))$ cannot likewise be virtually contained in the normal subgroup generated by the  $p$-th powers of Dehn twists for $p\ge 5$ odd or in the larger subgroup which is the kernel of the level $p$ Reshetikhin-Turaev quantum representation,    unless $\phi$ is isotrivial.

\subsection{Definitions}
All algebraic varieties and stacks considered in these notes will be over $\C$ and we will mainly think of them through their analytification as complex-analytic objects. We use the same notations and conventions as in \cite{mcgorb}. 

Let $g,n \in \N$ such that $2-2g-n<0$ and let $p\in \N^*$ be a positive integer. 
We denote by $\Mgnbar$ the moduli stack of genus $g$ stable curves with $n$ marked points and by 
$\overline{M_{g}^n}$ its coarse moduli space. We omit $n$ in the notation when $n=0$. 
Further   $\mathcal{M}^{n}_{g}$ will denote the moduli stack of smooth curves of genus $g$ and $n$ punctures and
$M_{g}^n\subset \overline{M_{g}^n}$ the Zariski open subset consisting of isomorphism classes of smooth curves.
The stacks $\mathcal{M}^{n}_{g}$ are separated, smooth, Deligne-Mumford and their moduli spaces are quasiprojective \cite{Knu,Knu2}, 
but non proper except in the trivial case  $g=0,n=3$. 

Set  $\Mod(\Sigma_g^n)$ for the mapping class group of the 
genus $g$ orientable surface $\Sigma_g^n$ with $n$ punctures (or marked points) and 
$\PMod(\Sigma_g^n)$ for the {\em pure} mapping class group consisting of the isotopy classes which fix 
{\em pointwise} the punctures. Then $\PMod(\Sigma_g^n)$ and  $\Mod(\Sigma_g^n)$ also occur as the 
fundamental groups of the analytification $\mathcal{M}^{n \ an}_{g}$ of the moduli stack of curves (see \cite{DMu}) and 
$[S_n\backslash{\mathcal{M}^{n \ an}_{g}}]$, respectively, where $S_n$ acts on $\mathcal{M}^{n \ an}_{g}$ by permuting the markings.

There are only finitely many, say $N_{g,n}$ 
conjugacy classes of Dehn twists, or equivalently, 
distinct orbits of essential simple closed curves on $\Sigma_g^n$ under the mapping class group action.
For instance $N_{g}=\lfloor \frac{g}{2}\rfloor+1$, $N_{g,1}=g$. 
Simple closed curves orbits are determined by the homeomorphism type of the 
complementary subsurface, which might be either connected for non-separating curves 
or disconnected and hence determined by the set/pair of genera of its two components, for separating curves. 
Fix an enumeration of these homeomorphism types starting with the non-separating one. For each  vector of positive integers $\mathbf k=(k_0,k_1,k_2,\ldots,k_{N_{g,n}-1})$ we define
$\Mod(\Sigma_g^n)[\mathbf k]$ as the (normal) subgroup generated 
by $k_0$-th powers of Dehn twist along non-separating simple closed curves and 
$k_j$-th powers of Dehn twists along  simple closed curves of type $j$.
As a shortcut we use $\Mod(\Sigma_g^n)[k;m]$ for $\mathbf k=(k,m,m,m\ldots,m)$,
$\Mod(\Sigma_g^n)[k]$ for $\mathbf k=(k,k,\ldots,k)$ and
 $\Mod(\Sigma_g^n)[k;-]$ for $\mathbf k=(k;)$, where $k_i$ are absent for $i>0$.

For every ${\mathbf{k}} \in \mathbb{N}_{>0}^{N_{g,n}}$ the quotient $\PMod(\Sigma_{g}^n)\slash \Mod(\Sigma_{g}^n)[\mathbf k]$  
is the fundamental group of a smooth proper Deligne-Mumford stack ${\overline{\mathcal{M}_{g}^{n \ an}}}[\mathbf k]$ compactifying $\mathcal{M}_{g}^{n \ an}$
 whose coarse moduli space is the moduli space  
${\overline{{M}_{g}^{n \ an}}}$ of stable $n$-punctured curves of genus $g$, hence is projective by \cite{Knu2}.

Let $K_i$ denote the kernel of the homomorphism $\PMod(\Sigma_{g}^n)\to \PMod(\Sigma_{g}^{n-1})$ induced by the forgetful map
$\Sigma_g^n\to \Sigma_g^{n-1}$ which omits the $i$-th puncture. By Birman's exact sequence, $K_i$ is isomorphic to $\pi_1(\Sigma_g^{n-1})$, when $n\geq 1$ and  $\Sigma_g^{n-1}$ is neither a  sphere, a torus, a 1-punctured nor a 2-punctured sphere. We call $K_i$ the {\em geometric} surface subgroups of  $\PMod(\Sigma_{g}^n)$.

Recall that an algebraic family $f:C \to \Mg^n$ is {\em non-isotrivial} if and only $f_*(\pi_1(C))$ is infinite. 

\subsection{Restricting unitary representations}
The first result of this article is: 

\begin{theoi}\label{unitary}
Let $g\geq 3$, $\rho$ be a finite rank unitary  representation of $\PMod(\Sigma_{g}^n)$ with 
infinite image. When $n\geq 1$ we assume furthermore that the restriction of $\rho$ to every geometric surface subgroup is also infinite. Then  for every  non-isotrivial algebraic family  $f:C \to \Mg^n$ of $n$-punctured smooth genus $g$ curves the image $\rho \circ f_*(\pi_1(C))$ is infinite. In particular, the restriction of an infinite unitary representation of finite rank of the mapping class group of a closed surface to an algebro-geometric subgroup should be infinite, when $g\geq 3$.
\end{theoi}

\begin{remark}
Actually the same result holds more generally for any semi-simple finite dimensional representation of $\PMod(\Sigma_{g}^n)$ which factors through some quotient 
$\PMod(\Sigma_{g}^n)\slash \Mod(\Sigma_{g}^n)[p]$. 
\end{remark}

This result can be effectively used to constrain monodromy groups of non-isotrivial holomorphic maps $f:C\to \Mg^n$  of smooth $n$-pointed curves  of genus $g$, defined on a hyperbolic Riemann surface $C$. Shiga has proved in (\cite{Shi}, Thm. 1) that $f_*(\pi_1(C))$ is irreducible and further Daskalopoulos and Wentworth (\cite{DW}, Thm. 5.7) improved this for $n=0$ to the  effect that  $f_*(\pi_1(C))$ is a sufficiently large subgroup of $\PMod(\Sigma_g)$, namely it contains two pseudo-Anosov mapping classes with distinct fixed points in the space of projective measured laminations. This condition is sufficient to realize the monodromy homomorphism
by a smooth equivariant harmonic map between $\widetilde C$ and the Teichm\"uller space (\cite{DW}, Cor. 5.5) in the case when 
$C$ is compact. 

\subsection{Restricting Reshetikhin-Turaev representations}
Quantum representations provide a large supply of linear representations which can be used to this purpose. 
Consider the Reshetikhin-Turaev  quantum representations $\rho_{g,p, (\mathbf{i})}$ (in the version of \cite{BHMV},  notation of \cite[Section 3.1]{mcgorb})  for all possible colors of the boundary:
\[ \rho_{g,p}:=\prod_{\mathbf{i}\in \mathcal{C}_p^n}\rho_{g,p, (\mathbf{i})}\]
It is known that $\Mod(\Sigma_{g}^n)[p]$ is contained in the kernel of $\rho_{g,p}$ and thus 
$\rho_{g,p}$ are representations of $\pi_1(\overline{\mathcal{M}_{g}^n}[p])=\PMod(\Sigma_{g}^n)\slash \Mod(\Sigma_{g}^n)[p]$, for odd $p\geq 5$. Moreover, images of 
Dehn twists have orders dividing $2p$ so that  $\Mod(\Sigma_{g}^n)[16p,2p]$ is contained in the kernel of $\rho_{g,p}$ and hence $\rho_{g,p}$ are representations of $\pi_1(\overline{\mathcal{M}_{g}^n}[16p,2p])=\PMod(\Sigma_{g}^n)\slash \Mod(\Sigma_{g}^n)[16p, 2p]$, for even $p\geq 10$, see \cite{mcgorb} for details.
The following statement is a consequence of Proposition \ref{theointro2} and was open  to the best of our knowledge: 

\begin{theoi} \label{corointro1} Let $g,n$ with $g\geq 2$. 
We assume that either $p\geq 5$ is odd or $p\ge 10$ is even and $(g,n,p)\neq (2,0,12)$. 
Then for every non-isotrivial algebraic family 
$f:C \to \Mg^n$  of smooth $n$-pointed curves  of genus $g$, the Reshetikhin-Turaev representation $\rho_{g,p}$ restricted to the image $f_*(\pi_1(C))$ has infinite image. In particular, under the same assumptions on ($g,p)$ the image of the fundamental group of a positive-dimensional algebraic subvariety of $\mathcal{M}_{g}$ is infinite under the Reshetikhin-Turaev representations $\rho_{g,p}$  of the mapping class group.
\end{theoi}

The infiniteness result above was known for some  particular curves $C$ thanks to the series  of papers \cite{F,LF,KoSa,LPS,Mas}. 
The strength of our result lies in that the surface group is not required to be a normal subgroup of the mapping class group, as in the previous works. 

By Shiga's result (\cite{Shi}, Cor. 1)  we know that  $f_*(\pi_1(C))$ must contain a pseudo-Anosov mapping class, because it is an irreducible subgroup. Thus, for $p$ large enough depending on the family $f$, the statement above is a consequence of the AMU conjecture (see \cite{AMU}). Although our conclusion is weaker than the AMU conjecture, it actually works uniformly for every allowed value of the level $p$ and for any algebraic family of curves.


Recently Godfard established  the existence of complex variations of Hodge structures whose monodromy is given by the Reshetikhin-Turaev representations $\rho_{g,p}$  
and for more general modular tensor category whose associated representations are semi-simple
(see \cite{Go3}). Using classical results of Corlette (\cite{Cor}) and Simpson (\cite{Sim}) we deduce:
\begin{coroi}
Under the assumptions of Theorem \ref {corointro1}, the Zariski closure of $\rho_{g,p}(f_*(\pi_1(C)))$ is a positive dimensional semi-simple group of Hodge type.
\end{coroi}

Whether the results above could be extended to all finite rank semi-simple representations of $\PMod(\Sigma_g^n)$ with infinite restrictions to geometric surface subgroups  is still unknown.  Moreover the question to decide if $\rho_{g,p}(f_*(\pi_1(C)))$ is an arithmetic group is quite challenging.

Note that the algebraicity assumption is necessary here, since for any $p$ there exist sufficiently large subgroups of $\Mod(\Sigma_g^n)[p]$, which are therefore contained in the kernel of $\rho_{g,p}$. 

\begin{coroi}\label{nonholomorphic}
Let $g,n$ with $g\geq 2$.  
Assume that either $p\geq 5$ is odd or $p\ge 10$ is even and $(g,n,p)\neq (2,0,12)$. 
Then for every non-isotrivial algebraic family 
$f:C \to \Mg^n$  of smooth $n$-pointed curves  of genus $g$, the image of $f_*(\pi_1(C))$
in  $\PMod(\Sigma_g^n)/\Mod(\Sigma_g^n)[p]$ is infinite.
\end{coroi}

\subsection{The Stein property}
Let ${\mathcal{T}_{g}^n}[\mathbf k]$ be the universal covering stack \cite{No2} of the analytification of $\overline{\mathcal{M}_{g}^n}[\mathbf k]$.  As an application of 
the previous results we obtain the following:

\begin{theoi}\label{theointro1} Let $g,n$ with $g\geq 2$.

\begin{enumerate}
   \item 
If $p\geq 7$ is odd, then ${\mathcal{T}_{g}^{n}}[p]$  is a Stein manifold.
\item 

If $p$ is even and $p\geq 14, p\not\in\{20,24\}$  then ${\mathcal{T}_{g}^{n}}[2p,  p/{\rm g.c.d.}(p,4)]$  is a Stein manifold.
 
\end{enumerate}

\end{theoi}

The theorem is proved in section \ref{pftheointro1} and improves on Theorem \ref{corointro1} above and the key Proposition  \ref{theointro2}.

\subsection{Teichm\"uller curves and Veech groups}
The theorem above applies to the Teichm\"uller curve attached to a Veech surface, as explained in section \ref{appteichcurves}. We then obtain:

\begin{coroi}\label{origami}
Let $\Sigma_g\to \Sigma_1$, $g\geq 2$, be a ramified covering which is branched over a single point of the torus $\Sigma_1$. Let $G\subset \Mod(\Sigma_g)$ be the subgroup of mapping classes 
of homeomorphisms of $\Sigma_g$ which lift homeomorphisms of $\Sigma_1$. 
Then the image of $\rho_{g,p}(G)$ is infinite for odd $p\geq 7$ 
and even  $p\geq 14, p\not\in\{20,24\}$.   
\end{coroi}

We now formulate a corollary in terms of groups generated 
by two multicurves to illustrate this point.  Recall that a multicurve $\underline{c}$ on the surface $\Sigma_g^n$ is  
the union of disjoint, essential simple closed curves. 
We allow a multicurve to contain several parallel copies of the same curve. 
If  $\underline{c}$ and $\underline{d}$ are two multicurves, we can isotope them in minimal position, namely such that $|\underline{c}\cap \underline{d}|=i(\underline{c},\underline{d})$, where $i(\underline{c},\underline{d})$ denotes the minimal number of intersection points between the two isotopy classes of multicurves. 
The configuration graph $\mathcal G(\underline{c}\cup \underline{d})$  has vertices associated to connected components of $\underline{c}$ and of $\underline{d}$ and edges associated to intersection points between the components. Leininger considered in \cite{Lein} the 
classical Dynkin graphs $A_n, D_n, E_6, E_7$ and $E_8$, which he  called {\em recessive} and the {\em critical} graphs  $P_{2n}, Q_{n}, R_7, R_8$ and $R_9$.


\begin{coroi}\label{corointrod2}
Let $G(\underline{c},\underline{d})$ be the subgroup of $\PMod(\Sigma_g^n)$ 
generated by the Dehn multitwists $T_{\underline{c}}$ and 
$T_{\underline{d}}$ with connected $\underline{c}\cup\underline{d}$ and 
$\underline{c}$ and $\underline{d}$ in minimal position. 
Assume that the configuration graph $\mathcal G(\underline{c}\cup \underline{d})$ is either critical or recessive. 
Then the image of $\rho_{g,p}(G(\underline{c},\underline{d}))$ is infinite, 
for odd $p\geq 7$ 
and even  $p\geq 14, p\not\in\{20,24\}$. 
\end{coroi}

\subsection{Degree two cohomology classes}
One of the main open problems in the theory of K\"ahler groups is a conjecture due to Toledo that the fundamental group $\Gamma$ of a compact K\"ahler manifold should satisfy $H^2(\Gamma, \mathbb{R}) \not =0$. This is obviously true for  aspherical compact K\"ahler manifolds. Reznikov observed this also holds if there is a complex Variation of Hodge structure $\mathbb V$ such that $H^1(\Gamma, {\mathbb V}) \not =0$.  In particular, thanks to Simpson's ubiquity theorem, (see \cite{Sim1}, \cite[Corollary 4.2, p.45]{Sim}) the conjecture holds true unless all finite rank complex  linear representations of $\Gamma$ are infinitesimally rigid. This very strong rigidity property follows from Margulis' superrigidity for $\Gamma$ a torsion free irreducible uniform lattice in an almost simple Lie group of Hermitian symmetric type of real rank at least 2; still the Toledo conjecture holds true in the latter case since the universal covering space is a bounded symmetric domain, hence is contractible. One may speculate that the K\"ahler groups we are studying here have the above rigidity property, but in any case they are non-trivial instances of the Toledo conjecture. 

We use a recent theorem due to  Dahmani (\cite{Da}) to give some evidence towards the Toledo conjecture for the stacks  $\overline{\mathcal{M}_{g}^n}[\mathbf k]$.
\begin{theoi}\label{H2}
 There exists some $p_0$ such that for $g\ge 2$ and $g.c.d (k_0, \ldots, k_{N_{g,n}})$  divisible by $p_0$ we have 
\[ H^2 (\Mod(\Sigma_g^n)/\Mod(\Sigma_g^n)[\mathbf k], \R) \not =0. \]
 \end{theoi}

\begin{rem}
The results of \cite{mcgorb} and of the present note extend to all TQFTs attached to a Rational Conformal Field Theory, or modular functor as in \cite{Go3}, assuming 
their restrictions to geometric surface subgroups are infinite. 
\end{rem}

\vskip 0.3cm {\bf Acknowledgements.}
We would like to thank Gavril Farkas for showing us the reference  \cite{GKM},  
Marco Boggi, Beno\^it Claudon,  Fran\c{c}ois Dahmani, Thomas Delzant, Pierre Godfard, Henri Gu\'enancia, Erwan Lanneau, Carl Lian,  Julien March\'e, Vasily Rogov and Yibo Zhang for interesting discussions related to some topics discussed in this note and the referee for useful comments improving the presentation. 

We dedicate this article to the memory of 
Domingo Toledo, whose work has inspired us for many years.

\section{Proof of Theorem \ref{unitary}}\label{pftheointro2}
The key ingredient is the fact that the Shafarevich morphisms for the uniformizable stacky compactifications of moduli spaces of curves endowed with a semi-simple representation 
of the fundamental group is a birational morphism (see Proposition  
\ref{theointro2}). This is a consequence of deep results of Gibney, Keel and Morrison 
describing all fibrations of the Deligne-Mumford compactifications of moduli space of curves
(see Theorem \ref{GKMtheorem}). Therefore algebraic curves in the moduli space cannot be contracted by the corresponding unitary representation of the mapping class group.  
  
\subsection{Stacky compactifications of moduli space of smooth curves}
For every ${\mathbf{k}} \in \mathbb{N}_{>0}^{N_{g,n}}$ the quotient $\PMod(\Sigma_{g}^n)\slash \Mod(\Sigma_{g}^n)[\mathbf k]$  
is the fundamental group of a smooth proper Deligne-Mumford stack ${\overline{\mathcal{M}_{g}^{n \ an}}}[\mathbf k]$ compactifying $\mathcal{M}_{g}^{n \ an}$
 whose coarse moduli space is the moduli space  
${\overline{{M}_{g}^{n \ an}}}$ of stable $n$-punctured curves of genus $g$, hence is projective, see \cite{mcgorb}.

 Recall that a smooth Deligne-Mumford stack $\mathcal{X}$  is {\em uniformizable} if $\mathcal{X}^{an} \cong [X/G]$ where the finite group $G$ acts on $X$ a smooth complex manifold (the action need not be effective).
  If $\mathcal{X}$  is uniformizable then its analytification is equivalent to a quotient stack of a properly discontinuous 
  action of its fundamental group on a simply connected complex manifold  $\mathcal{X}^{an} \cong \widetilde{\mathcal X^{univ}} /\pi_1(\mathcal{X})$, where  $\widetilde{\mathcal X^{univ}}$ is the universal covering space of $\mathcal{X}^{an}$.
  
  The following facts were proved in \cite{mcgorb}:
  \begin{prop}\label{eyfu}
  \begin{enumerate}
   \item 
If $p\geq 5$ is odd, ${\overline{\mathcal{M}_{g}^{n \ an}}}[p]$  is uniformizable.
\item 

If $p\ge 12$ is even and $(g,n,p)\neq (2,0,12)$, $\overline{{\mathcal M}_g^{n \; an}}[2p,  p/{\rm g.c.d.}(p,4)]$ is uniformizable.
  \end{enumerate}
\end{prop}

A key ingredient in the sequel is the notion of Shafarevich morphism: 
\begin{definition} \label{defshafmorph} 
Let $\mathcal{X}$ be a connected smooth proper Deligne-Mumford stack and $\rho: \pi_1(\mathcal{X}) \to GL_N(\C)$ be a finite rank semi-simple complex linear representation, namely which has a Zariski dense image in a reductive subgroup of $GL_N(\C)$. A Shafarevich morphism associated to $(\mathcal X, \rho)$ 
is  a surjective map of stacks $s_{\rho}:\mathcal{X} \to \mathcal{S}_{\rho}$ 
with connected fibers, onto a uniformizable normal proper Deligne-Mumford stack $\mathcal{S}_{\rho}$, with the following property:
for any  map  $Z \buildrel{f}\over{\to} \mathcal{X} $ from  a   connected algebraic variety $Z$ the image 
$\rho(f_* \pi_1(Z))$  is finite if and only if $ s^{mod}_{\rho}\circ f (Z)=\{pt\}$,  where $s^{mod}_{\rho}: \mathcal{X}^{mod} \to \mathcal{S}^{mod}_{\rho}$ denotes the moduli map.
\end{definition}

The first author proved that Shafarevich morphisms exist for smooth complex projective varieties (see \cite[Théorèmes 2, 2.1.7] {E6}). We will need in the sequel the existence 
of Shafarevich morphisms more generally, for smooth proper uniformizable Deligne-Mumford stacks.  We postpone the discussion on stacks and the proof to the next section. Assuming this,  our main technical tool is the following result:

\begin{prop} \label{theointro2} Suppose that $\overline{\mathcal{M}_{g}^n}[\mathbf k]$ is uniformizable. 
Let $g\geq 2$ and $\rho$ be a  semi-simple finite rank complex linear representation of 
$\pi_1(\overline{\mathcal{M}_{g}^n}[\mathbf k])=\PMod(\Sigma_{g}^n)\slash \Mod(\Sigma_{g}^n)[\mathbf k]$  with  infinite image. 
Assume that $\rho$ restricts to infinite representations on all geometric surface subgroups, when $n\geq 1$.  
\begin{enumerate}
\item Then the Shafarevich morphism  $s^{mod}_{\rho}:\overline{M_{g}^n}\to {S}_{\rho}^{mod}$
is a birational contraction whose exceptional locus lies in the boundary. 
\item If $C$ is a curve, then for every  non-isotrivial algebraic family  $f:C \to \Mg^n$ of $n$-punctured smooth genus $g$ curves the image $\rho \circ f_*(\pi_1(C))$ is infinite.
\end{enumerate}
\end{prop}
The proof, postponed at the end of this section, is a simple argument based on results on the geometry of 
the moduli space of curves from \cite{GKM} and the basic properties of the Shafarevich morphism \cite{E6, CCE}. 
We think that the uniformizability  assumption is not  essential  in Proposition \ref{theointro2}, see Remark \ref{unifdrop}. 

Note that the corollary does not apply to the homology representation in
$\mathrm{Sp}(2g, \Z/ k_0 \Z)$ which  has finite image since it takes value in a finite group and in this case $\mathbf{k}=[k_0;-]$.

\begin{proof}[End of proof of Theorem \ref{unitary}]
According to \cite[Corollary 2.7]{AS}, when $g\geq 3$ every  finite rank unitary (or, more generally, unipotent-free) representation of  
$\PMod(\Sigma_{g}^n)$ factors through some quotient  
$\PMod(\Sigma_{g}^n)\slash \Mod(\Sigma_{g}^n)[p]$, for some positive integer $p$. 

By passing to a tuple $\mathbf{k}$ such that  $p$ divides all its components, if needed, we can assume  that $\overline{\mathcal{M}_{g}^n}[\mathbf{k}]$ is uniformizable thanks to Proposition \ref{eyfu}.
Therefore we can apply the result of Proposition \ref{theointro2} to complete the proof of Theorem \ref{unitary}. 
\end{proof}

\begin{remark}
The unitarity assumption in Theorem \ref{unitary} is only needed to insure that the representation factors through some quotient   
$\PMod(\Sigma_{g}^n)\slash \Mod(\Sigma_{g}^n)[p]$, for some positive integer $p$. 
The same proof shows that the result holds, more generally, for finite rank semi-simple representations of some quotient $\PMod(\Sigma_{g}^n)\slash \Mod(\Sigma_{g}^n)[p]$.
However, if we aim at extending Theorem \ref{unitary}  to all  finite rank semi-simple representations $\rho$ of $\PMod(\Sigma_{g}^n)$ we need a version of 
Proposition \ref{theointro2} working for the (non-proper) stacks $\Mg^n$, which  
seems rather difficult.
\end{remark}

\subsection{Shafarevich morphisms for stacks}

Let $\mathcal{X}$ be a connected smooth proper Deligne-Mumford stack. Let $\rho: \pi_1(\mathcal{X}) \to GL_N(\C)$ be 
a semi-simple finite rank complex linear representation.

\begin{prop} \label{shafmorph} Assume that $\mathcal{X}$ is uniformizable. Then there exists a unique up to equivalence 
Shafarevich morphism $s_{\rho}:\mathcal{X} \to \mathcal{S}_{\rho}$. 
\end{prop}
\begin{rem}\label{shafcurve}
One could restrict  in Definition \ref{defshafmorph} to the case $Z$ is a smooth algebraic curve. 
 \end{rem}

\begin{proof} First of all if $\mathcal{X}$ is not a stack but a compact K\"ahler manifold  this follows from the existence of the Shafarevich morphism for $\rho$, see
\cite[Proposition 3.14]{CCE}, by the very definition of the Shafarevich morphism.
As the existence of the Shafarevich morphism is bimeromorphically invariant, 
the same holds if $\mathcal{X}$ is  bimeromorphic to a compact K\"ahler manifold, 
in particular if it is Moishezon, hence if $\mathcal{X}$ is representable stack.

Choose $ \psi: X \to \mathcal{X}$ a finite uniformization so that $X$ is a smooth complex algebraic space, with an action of a finite group $G$, so that $\mathcal{X}\simeq [X/G]$. 
Applying the previous remark,  we construct the Shafarevich morphism attached to the restriction $\rho'$ of $\rho$ to $\pi_1(X)$,  say $Sh_{\rho'}: X \to Sh_{\rho'}(X)$. 
The action of $G$ descends and we get a map of stacks $\mathcal{X} \to \mathcal{S}_{\rho}=[  Sh_{\rho'}(X)/G]$ which has the required property. 
\end{proof}

 \begin{rem}\label{rem:proj}
With notation from Definition \ref{defshafmorph}, let us further assume that 
$Z$ is smooth and proper. Then $f^*\rho$ is a finite rank semi-simple complex linear representation (see \cite{Cor}). 
\end{rem}

\begin{lem} If $\mathcal{X}$ has a projective moduli space so has $\mathcal{S}_{\rho}$.
 \end{lem}

 \begin{proof}
The projectivity of $\mathcal{S}_{\rho}$  follows from the projectivity of $Sh_{\rho}(X)$ when $X$ is a complex projective manifold \cite{CCE}. The details being omitted there, let us give the idea of the proof. The  normal complex space
$Sh_{\rho}(X)$ comes equipped with a holomorphic line bundle
that satisfies the Nakai-Moishezon criterion for ampleness \cite{E6}.
A compact complex space having such a holomorphic line bundle is projective. Indeed by  Siu's solution of the Grauert-Riemenschneider conjecture, it is Moishezon hence an algebraic space. One then applies \cite[Theorem 3.11]{Kol}. One can alternatively adapt Koll\'ar's proof of the Nakai-Moishezon criterion for algebraic spaces to the complex-analytic context. A somewhat delicate part is that a line bundle on a compact complex space is ample if and only if its pull-back by a finite surjective morphism is ample.
\end{proof}

\begin{lem} \label{fonctoshaf}
 If $\phi:\mathcal{X}'\to \mathcal{X}$ is a map of uniformizable connected smooth proper Deligne-Mumford stacks and $\rho$ is as above, there is a natural map $\mathcal{S}_{\phi^*\rho}\to \mathcal{S}_{\rho}$ such that 
 $\mathcal{S}^{mod}_{\phi^*\rho}\to \mathcal{S}^{mod}_{\rho} $ is a finite morphism.
\end{lem}
\begin{proof}
 This functoriality of the Shafarevich morphism is an easy consequence of the definition, see \cite{E6} for the case of compact projective manifolds. 
\end{proof}

\begin{rem} \label{unifdrop} For smooth Deligne-Mumford orbifolds with a projective moduli space, one can probably drop the uniformizability condition in Proposition \ref{shafmorph} and Lemma \ref{fonctoshaf}. 
One has  to use \cite[Prop. 1.16]{ajm} and redo the construction of the Shafarevich morphism in \cite{E6} working on the  manifold  $\mathcal{X}'$ which is endowed with a properly discontinuous
action of $\pi_1(\mathcal{X})$. However, the existing litterature is restricted to a less general setting. Since we have no convincing application to Theorem \ref{theointro2} when 
the uniformizability assumption is not satisfied, we leave it to the interested reader. 
\end{rem}

\subsection{Application to ${\overline{\mathcal{M}_{g}^{n}}}[\mathbf k]$ }
In \cite{GKM} Gibney, Keel and Morrison described all fibrations of  ${\overline{{M}_{g}^{n}}}$  to projective varieties. Recall that a 
{\em fibration} means a proper surjective morphism with  geometrically connected fibers.  
Specifically, they proved: 
\begin{theorem}[\cite{GKM} Corollary 0.10, 0.11]\label{GKMtheorem}
If $g\geq 2$ and $n\geq 1$, then any fibration of ${\overline{{M}_{g}^{n}}}$ to a projective variety factors through a projection to ${\overline{{M}_{g}^{j}}}$ for some $j<n$, while ${\overline{{M}_{g}}}$ has no fibrations. Moreover, if $g\geq 1$, then any birational morphism from ${\overline{{M}_{g}^{n}}}$ to a projective variety 
has exceptional locus contained in $\partial {\overline{{M}_{g}^{n}}}$. 
\end{theorem}

We now want to analyse the Shafarevich morphisms associated to $\mathcal{X}={\overline{\mathcal{M}_{g}^{n}}}[\mathbf k]$. Henceforth  we  assume $\overline{\mathcal{M}_{g}}[\mathbf k]$ is uniformizable,  and fix $\rho: \pi_1({\overline{\mathcal{M}_{g}^{n}}}[\mathbf k]) \to GL_N(\C)$ a semi-simple representation. 

Let $C_i\subset \mathcal{M}^n_{g} $ be the algebraic curve which appears as a fiber of the (representable) $i$-th forgetful map $\mathcal{M}^n_{g} \to \mathcal{M}^{n-1}_{g}$. Recall that $\pi_1(C_i)$ identifies with $K_i$, when $g\geq 2$, from the Birman exact sequence.

\begin{proposition} \label{gibney}
Whenever  $\pi_1(\overline{\mathcal{M}^n_{g}}[\mathbf k])$, $g\geq 2$, has a finite rank complex linear representation $\rho$ with an infinite image on all $\pi_1(C_i)$ $i=1, \ldots, n$, then the Shafarevich morphism $s^{mod}_{\rho}:\overline{M^n_{g}}\to \mathcal{S}^{mod}_{\rho}$
is a birational contraction whose exceptional locus lies in the boundary.
 \end{proposition}

\begin{proof}
By Remark \ref{rem:proj} $\mathcal{S}^{mod}_{\rho}$ is a projective variety. 
According to Theorem \ref{GKMtheorem} above every fibration ${\overline{{M}_{g}^{n}}}\to V$ to a projective variety $V$ is the  composition of a birational morphism ${\overline{{M}_{g}^{j}}}\to V$  with 
the tautological projection ${\overline{{M}_{g}^{n}}}\to {\overline{{M}_{g}^{j}}}$ for some $j<n$. Thus 
the morphism $s_{\rho}^{mod}: \overline{M_{g}^{n}} \to \mathcal{S}^{mod}_{\rho}$  is either a birational morphism whose exceptional locus lies in $\partial \overline{M_{g}^{n}}$ or factors through
 one of the $n$ natural forgetful maps $\overline{M_{g}^{n}} \to \overline{M_{g}^{n-1}}$. 
Observe that $s_{\rho}^{mod}$ cannot be the constant morphism since $\rho$ has an infinite image.  When $n\geq 1$ the curves $C_i$ are not contracted by $s_{\rho}^{mod}$ because $\pi_1(C_i)$ have infinite images 
 by $\rho$, according to our assumptions. 
\end{proof}

\begin{proof}[Proof of Proposition \ref{theointro2}] 
It follows from  Remark \ref{shafcurve} and Proposition \ref{gibney} that $s_{\rho}^{mod}$ is a birational morphism and an isomorphism on ${M_{g}^{n}}$. Therefore the curve $C$ is not contracted by $s_{\rho}^{mod}$ and hence $\rho(\pi_1(C))$  is infinite by the definition \ref{shafmorph} of the Shafarevich morphism. 
\end{proof}

\section{Proof of Theorem \ref{corointro1}}
It suffices to show that (linearized versions of) Reshetikhin-Turaev representations satisfy 
the assumptions of Proposition \ref{theointro2}. To this purpose we need to prove some infiniteness statements for the quantum representations considered. 
Specifically, we shall prove: 

\begin{prop}\label{infiniteness} Let  $g,n$ with $2g-2+n>0$ and 
$p\not\in\{1,2,3,4,5,6,8,10,12,20,24\}$. Then 
there exists some quantum representation of the groups 
$\Mod(\Sigma_{g}^{n})/\Mod(\Sigma_g^n)[p]$, for odd $p$ and 
$\Mod(\Sigma_{g}^{n})/\Mod(\Sigma_g^n)[2p,  p/{\rm g.c.d.}(p,4)]$ for even $p$, respectively, 
whose image is infinite.
\end{prop}
It is enough to consider the case when 
$(g,n)\in\{(1,1), (0,4)\}$. The case $(0,4)$ was treated in \cite{F,Mas}.
Assume from now on that $(g,n)=(1,1)$. 

Let $B_3$ denote the braid group on two strands, with the usual presentation in its standard 
generators:  
\[ B_3 =\langle \sigma_1, \sigma_2\; |\; \sigma_1\sigma_2\sigma_1=\sigma_2\sigma_1\sigma_2\rangle\]
There is a surjective homomorphism $B_3\to \Mod(\Sigma_1^1)$, sending the 
standard generators $\sigma_1$ and $\sigma_2$ into the Dehn twists $T_a$ and $T_b$ respectively. Here $a$ is a meridian and $b$ a longitude of the torus $\Sigma_1^1$. 
By composing the quantum representation $\rho_{1, p, (i)}$ of $\Mod(\Sigma_1^1)$ with this surjection we obtain a projective representation of $B_3$. As $H^2(B_3)=0$, there is a 
linear lift of this projective representation. However, this lift depends on the  
choice of the lifts $t$ and $t^*$ of $\rho_{1, p, (i)}(T_a)$ and 
$\rho_{1, p, (i)}(T_b)$, respectively. A particular lift was defined in 
(\cite{GM2}, Prop. 11.7), although this was implicit in earlier work as \cite{MR}.

For $p\in \N^*$ we denote by $v_2(p)\in \N^*$ the 2-adic valuation, i.e.
largest integer such that $\frac{p}{2^{v_2(p)}}$ is an odd integer.  
We now prove the following lemma, which implies the claimed result for odd $p$:
\begin{lem}\label{infinitenessodd} If 
$\frac{p}{2^{v_2(p)}}\geq 7$, then  there exists $(\mathbf{i})\in \mathcal{C}_p^n$ such that $\rho_{1,\frac{p}{2^{v_2(p)}},(\mathbf{i})}(\PMod(\Sigma_{1}^{1}))$ is infinite.
\end{lem}
\begin{proof}
The space $W_{1,p, (i)}$ has a basis given by $p$-admissible colorings of the tadpole graph with one tail labeled $i$. Thus an element of this basis  
is determined by the color $a$ on the loop edge.  
Note that $t$ and $t^*$ have the eigenvalues $(-1)^aA^{a(a+2)}$, where $a$
is belongs to the $p$-admissible colorings (see \cite{GM2,MR}).

Let first $p$ be odd, $p\geq 7$. 
The set of colors is then $\mathcal{C}_p=\{0,2,4,\ldots,p-3\}$.  
We set the color $i=p-5>0$. 
Then the space of conformal blocks $W_{1,p, (p-5)}$ has dimension 2. Indeed, 
when $p=4k+1$ the color $a$ of the loop edge in a $p$-admissible coloring 
of the tadpole with one tail labeled $p-5$  takes the values $a\in\{2k-2,2k\}$ and 
when  $p=4k+1$ the color $a$ takes the values $a\in\{2k,2k+2\}$.

By direct calculation, or using \cite[Propositions 2.1 and 2.3]{FK} and the fact that the image is not 
abelian, we derive that the 2-dimensional linear representation of $B_3$ is 
the Burau representation twisted by a character. 
The eigenvalues of a Dehn twist along a nonseparating simple closed curve 
in the linear lift of the quantum representation are $A^{a(a+2)}$, where $a$ is 
$p$-admissible. 
As the Burau representation at the root of unity $q$ has eigenvalues 
$1$ and $-q$, it follows that the Burau representation factor arising above 
is the one evaluated at $q=-A^{-2}$, if $p=4k+1$ and $q=-A^2$ if $q=4k+3$.  
Since $A$ is a primitive $2p$-th root of unity we derive that 
$-q$ is a primitive $p$-th root of unity.  According to \cite[Propositions 3.2 and 3.4]{FK} 
the image of the Burau representation of a suitable pure braid subgroup at the negative of a $p$-th 
root of unity is an infinite triangle group, as soon as $p\not\in\{2,3,4,5,6\}$. 
This implies the claim for odd $p$.

When $p$ is even  and $q$ is the maximal odd divisor $q=\frac{p}{2^{v_2(p)}}$
we take $i=q-5$. Then the image of $\rho_{g,q,(\mathbf{i})}(\PMod(\Sigma_{g}^{n}))$ is infinite, as soon as $q\geq 7$.  
\end{proof}

To deal with the even case, we first prove: 

\begin{lem}\label{infinitenesseven} If $p\equiv 0\; ({\rm mod}\; 4)$ and 
$p$ does not divide $120$, then  $\rho_{1,p,(p-6)}(\PMod(\Sigma_{1}^{1}))$ is infinite.
\end{lem}
We improve the strategy used in \cite{Kor}, where the result is proved for 
particular families of $p$. The key ingredient is the following lemma due to Coxeter: 

\begin{lemma}\label{coxeter}
A group which admits an irreducible representation on a finite dimensional 
vector space $V$ preserving an indefinite Hermitian form  
should be infinite. 
\end{lemma}

\begin{proof}[Proof of Lemma \ref{infinitenesseven}]
Recall that for even $p$ the set of colors $\mathcal C_p=\{0,2,4,\ldots,\frac{p-4}{2}\}$. 
The $p$-admissibility conditions implies that the boundary color is even, say $i=2c$. 
Consider $p=4k$ and $i=2k-6$. The space of conformal blocks $W_{1,p, (2k-6)}$ is then 
of dimension $5$, with a basis $u_s, s\in\{0,1,2,3,4\}$ where $u_s$ corresponds to the 
colored tadpole graph whose loop edge is labeled $s+k-3$.
The vectors $u_s$ are eigenvectors for the matrix $t$. After rescaling the 
$B_3$ representation by a factor $(-A)^{-k^2+1}$, the eigenvalues of 
$t$ are respectively  $\lambda_0=-\zeta^4, \lambda_1=\zeta, \lambda_2=1, \lambda_3=-\zeta$ and $\lambda_4=-\zeta^4$, where 
$\zeta=A^{2k+1}$ is also a primitive $2p$-th root of unity. 

The Hermitian form  $\langle\; ,\;\rangle$ invariant by the linear lift of the 
quantum representation was computed in \cite{BHMV} and we have: 
\[ \frac{\langle u_{s+1}, u_{s+1}\rangle}{\langle u_s, u_s\rangle}= \frac{[2k-4+s][s+1]}{[k-1+s][k-2+s]}, \; s\in\{0,1,2,3\}, \]
where the quantum integer $[n]$ is defined as 
\[ [n]=\frac{A^{2n}-A^{-2n}}{A^2-A^{-2}}\]
Let $A=\exp\left(\frac{2\\pi \sqrt{-1} \ell}{2p}\right)$, where $\ell$ is odd.  
By direct computation we have for any $\ell$ that: 
\[ \frac{\langle u_{1}, u_{1}\rangle}{\langle u_0, u_0\rangle} > 0, \frac{\langle u_{4}, u_{4}\rangle}{\langle u_3, u_3\rangle} >0.\]
On the other hand 
\[ \frac{\langle u_{2}, u_{2}\rangle}{\langle u_1, u_1\rangle} = 4\sin\left(\frac{3\pi \ell}{2k}\right) \cos\left(\frac{\pi \ell}{2k}\right)\sin\left(\frac{\pi \ell}{4k}\right),
 \]
 \[ \frac{\langle u_{3}, u_{3}\rangle}{\langle u_2, u_2\rangle} = 2\sin\left(\frac{3\pi \ell}{2k}\right) \cos\left(\frac{\pi \ell}{2k}\right).
 \]
These quantities are both negative 
when $2k >\ell> \frac{4}{3}k$. In this case the Hermitian form has signature 
$(+,+,-,+,+)$. 
Therefore, we can choose the primitive $2p$-th root of unity $A$ such that  the Hermitian form  $\langle\; ,\;\rangle$ is indefinite, when $k\geq 4$.

Suppose from now on that $p$ does not divide $120$. 

If the 5-dimensional representation of $B_3$ on $W_{1,p, (2k-6)}$ is irreducible, then Coxeter's lemma \ref{coxeter}
permits to conclude that its image is infinite. 

Assume that the representation above is not irreducible and let $V\subset W_{1,p, (2k-6)}$
be an invariant subspace, of dimension $r$. By passing to the orthogonal of $V$ if necessary, we can assume that $r\in\{1,2\}$. 
 
The element $(tt^*)^3$ acts as a scalar $\delta$ on 
$W_{1,p, (2k-6)}$, because it is a lift of the Dehn twist along the curve encircling once the puncture (see \cite{GM2}, Cor. 11.10 for its exact value). Since $\det (tt^*)^3=\delta^5$ we derive the following equation: 
\[\delta^5 =(\lambda_0\lambda_1\lambda_2\lambda_3\lambda_4)^6.\]
 
Consider first $r=1$ and let the eigenvalue of $t$ and $t^*$ corresponding to the subspace $V$ be $\lambda_{i}$. Then $(t|_{V}t^*|_{V})^3$ acts as the scalar $\delta$, so that 
\[  \delta =\lambda_i^6.\]
Replacing this value of $\delta$ in the previous equation we obtain the identity:   
\[ (\lambda_0\lambda_1\cdots \lambda_4)^{6}=(\lambda_{i})^{30}\]
This implies $\zeta^{60}=1$, which contradicts our assumptions on $p$. 

Consider now that $r=2$ and let that the eigenvalues of $t$ and $t^*$ corresponding to the subspace $V$ be $\lambda_{i}$ and $\lambda_j$, where $i<j$. 
By computing the determinant of the 2-by-2 scalar matrix $(tt^*)|_V$  we derive that
\[  \delta^2 =(\lambda_i\lambda_j)^6.\]
Replacing this value of $\delta$ in the first equation above we obtain the identity: 
\[ (\lambda_0\lambda_1\cdots \lambda_4)^{12}=(\lambda_{i}\lambda_j)^{30}\]
This relation implies that either $\zeta^{120}=1$, or else 
$(i,j)\in\{(0,2), (2,4), (1,3)\}$. 

When $(i,j)=(0,2)$, the restriction of the representation of $B_3$ to $V=\C u_0\oplus \C u_2$ is irreducible and the restriction of the Hermitian form  $\langle\; ,\;\rangle$  to $V$ is indefinite, for a suitable choice of the root $A$. 
When $(i,j)=(2,4)$, the situation is symmetric.  

Eventually, if  $(i,j)=(1,3)$, the restriction of the representation of $B_3$ to $V^{\bot}=\C u_0\oplus \C u_2\oplus \C u_4$ is irreducible and the restriction of the Hermitian form  $\langle\; ,\;\rangle$  to $V^{\bot}$ is again indefinite, for a suitable choice of the root $A$. 

In all cases above the image of the representation should be infinite, again by the Coxeter 
lemma \ref{coxeter}.  
\end{proof}

\begin{proof}[End of proof of Proposition \ref{infiniteness}]
When $p\geq 7$ is odd or $p\equiv 2 \; ({\rm mod}\; 4)$
we use lemma \ref{infinitenessodd}. If $4$ divides $p$ then 
the cases excluded by both lemmas \ref{infinitenessodd} and  \ref{infinitenesseven} are those from the statement. 
\end{proof}

\begin{rem}
Lemma \ref{infinitenessodd} cannot be extended to $p=5$. Indeed 
$\rho_{1,5,(i)}(\Mod(\Sigma_1^1))$ is finite both when $i=0$ and $i=2$.
However $\rho_{g,5}(\Mod(\Sigma_g))$ is infinite for all $g\geq 2$. 
The kernel of $\rho_{g, 5}$ therefore provides an infinite index subgroup 
of $\Mod_g$ whose intersection with any subgroup $\Mod(\Sigma_{1,1})$ associated to a subsurface 
 $\Sigma_{1,1}$ of $\Sigma_g$ is of finite index. 
\end{rem}

\begin{rem}
For $p=10$, the image $\rho_{1,10,(i)}(\Mod(\Sigma_1^1))$ is also finite, as 
$\rho_{2,10}(\Mod(\Sigma_2))$ is finite (see \cite{F}). Note that 
$\rho_{g,10}(\Mod(\Sigma_g))$ is infinite, when $g\geq 3$. 
\end{rem}

\begin{rem}
For $p=6$ and $p=8$  the corresponding representations $\rho_{1,p,(i)}(\Mod_1^1)$ are also finite, as $\rho_{2,p,(i)}(\Mod_2)$ are known to be finite.  
For $p\in\{12, 20, 24\}$, we know that 
$\rho_{2,p}(\Mod_g)$ are also infinite for $g\geq 2$ while  
$\rho_{1,5,(i)}(\Mod_1^1)$ is finite, whenever $i\leq p-6$, as the Hermitian form $\langle \; ,\;\rangle$ is positive definite for all values of $A$. 
\end{rem}

\begin{rem}
When $p\equiv 0\; ({\rm mod}\; 4)$  
the image $\rho_{1,p,(p-4)}(\PMod(\Sigma_{1}^{1}))$ of the 3-dimensional representation of $B_3$  is finite. 
Indeed the representation is actually integral for all $p$ (see \cite{GM2}) 
and the Hermitian form $\langle \; ,\;\rangle$ is positive definite for all values of $A$, as it was noted in \cite{Kor}. 
\end{rem}

We further need the following result which is taking care of the images of geometric subgroups of mapping class groups: 
\begin{prop}
Let $g,n$ with $2g-2+n>0$, $n\geq 1$ and $(g,n)\neq (1,1)$. Then the image $\rho_{g,p}(K_i)$ of a geometric surface subgroup is infinite, when $p\not\in\{1,2,3,4,5,6,8,10,12,20,24\}$. 
\end{prop}
\begin{proof}
In the situation at hand there exists some  1-punctured pair of pants  embedded in $\Sigma_g^n$ which is essential, namely the homomorphisms between fundamental groups is injective. 
It is known that the image of the fundamental group of the pair of pants by some quantum representation of 
$\Sigma_g^n$ is infinite non-abelian. This is shown in the proof of Prop. 3.2 in \cite{LF} for odd $p$, and the same arguments work for even $p$ in the given range.  This implies that the image of the fundamental group of 
$\Sigma_g^{n-1}$ by the same representation is infinite as well, as claimed.  
\end{proof}

Note that for $g\geq 2$ the mapping class group representations $\rho_{g,p}$ are projective and cannot be linearized.
It is well-known that $\rho_{g,p}$ lift to linear representations of a central extension of the mapping class group 
by some finite cyclic group depending on $p$.  By composing the adjoint representation
of the projective linear group with $\rho_{g,p}$ we obtain linear representations 
which have infinite images precisely when $\rho_{g,p}$ have infinite images. 

Eventually recall that $\rho_{g,p}$ factors through $\pi_1(\overline{\mathcal{M}_{g}^n}[p])$ for odd 
$p$ and $\pi_1(\overline{\mathcal{M}_{g}^n}[16p,2p])$ for even $p$, respectively. Indeed, it is known that the orders of the images of Dehn twists divide $2p$ for even $p$, see \cite{mcgorb} for a precise computation of the orders. 

Therefore the adjoint linear representations of $\rho_{g,p}$ satisfy the hypothesis of Proposition  \ref{theointro2}. 
This completes the proof of Theorem \ref{corointro1}.

\section{Proof of Theorem \ref{theointro1}\label{pftheointro1}}

\subsection{Idea of the proof and evidence from the F-Conjecture}

This conjecture  \cite{GKM} claims that a $\Q$-divisor on $\overline{M_{g}^n}$ is ample if and only if $D. \bar \Gamma>0$ for every $\Gamma\subset  \overline{M_{g}^n}$ a one-dimensional stratum 
of the Deligne-Mumford stratification of $\partial \overline{M_{g}^n}$. The Deligne-Mumford stratification is actually the stratification by the topological type of the stable curve.  

We denote by $X_p$ a complex projective manifold which is a finite Galois covering space of $\overline{M_{g}^n}$ in such a way that
if $p\geq 5$ is odd, $\mathcal{X}_p={\overline{\mathcal{M}_{g}^{n \ an}}}[p]$  
(resp. if $p\ge 12$ is even and $(g,n,p)\neq (2,0,12)$, $\mathcal{X}_p=\overline{{\mathcal M}_g^{n \; an}}[2p,  p/{\rm g.c.d.}(p,4)]$) is a quotient stack of $X_p$ by a finite group action.

For a smooth projective variety the sets 
of points of the moduli space of linear representations of its fundamental 
group, the moduli of flat semi-stable vector bundles and 
the moduli space of semi-stable Higgs bundles with trivial Chern classes are identified by the Riemann-Hilbert and Hitchin-Simpson correspondences, respectively. Moreover, the so-called Simpson ubiquity theorem (see \cite{Sim1}, \cite[Corollary 4.2, p.45]{Sim}) shows that fixed points for the $\C^*$ rescaling action on the moduli space of Higgs bundles correspond to the isomorphism classes of the holonomy representations of complex polarized variations of Hodge structures of weight zero. In particular, every semi-simple linear representation  of the fundamental group can be deformed to a complex variation of Hodge structures.

Constructible subsets of algebraic varieties are finite unions of locally closed sets. A subset of the moduli space of representation is {\em absolutely constructible} if it is identified with constructible subsets in the two other moduli spaces, equivariant with respect to the action of the Galois group $G_{\Q}$ over $\Q$ and $\C^*$-invariant in the Higgs bundle case, see \cite[section 6]{Sim2}, \cite[section 5.2]{E6}.

We denote by $R_p$ the absolute constructible subset on $X_p$ generated by (i.e. the smallest such set containing) the Galois conjugates of the pull back of the Reshetikhin-Turaev representations $\rho_{g, p, (\mathbf{i})}$ where $(\mathbf{i})=(i_1, \ldots, i_n) \in \mathcal{C}_p^n$
is a coloring of the $n$ punctures.

We look at its Shafarevich morphism $sh_{R_p}:X_p \to sh_{R_p}(X_p)$.

\begin{lem} \label{critstein} $sh_{R_p}$ is a birational contraction.  If  $sh_{R_p}=\mathrm{id}_{X_p}$,  then the universal covering space of $X_p$   is a Stein manifold.  
 \end{lem}
\begin{proof} The first statement follows from Proposition \ref{gibney}.  If $sh_{R_p}$ is the identity, 
 it follows from \cite{E6} that the covering space of $X_p$ attached to $R_p$  is a Stein manifold. Hence its universal covering space, which 
 coincides with that of $\mathcal{X}_p$, is Stein  since the universal covering space of a Stein manifold is Stein. 
\end{proof}

\begin{rem}
 Since the  $\rho_{g,  p, (\mathbf{i})}$ are in $R_p$ it is enough to show that $sh_{\rho_{g,p}}=\mathrm{id}_{X_p}$ since we have a factorisation $sh_{\rho_{g,p}}:X_p\buildrel{sh_{R_p}}\over{ \longrightarrow} sh_{R_p}(X_p) \to sh_{\rho_{g,p}}(X_p)$. 
\end{rem}

\begin{rem}\label{motivicity}
It is tempting  to conjecture that the absolute constructible set generated by the Reshetikhin-Turaev representations is discrete and that:
\[R_p=\{ \rho_{g, p,(\mathbf{i})}^{\sigma}, \quad (\mathbf{i} ) \in \mathcal{C}_p^n, \ \sigma\in G_{\Q} \}\]
If it were true, then Simpson's ubiquity theorem would imply that  the Galois  conjugates of the Reshetikhin-Turaev representations are complex variations of Hodge structures. This would also follow (see \cite[Theorem 5, p.59]{Sim})  if we knew that the Reshetikhin-Turaev representations $\rho_{g,p}$ were locally rigid. At present we only know that the local rigidity holds when $p=5$ and $g\geq 3$ or $p\geq 7$ is prime and $g\geq 7$ (see \cite{Go1,Go4}).
\end{rem}

\begin{rem}\label{cvhs}
Godfard (\cite{Go3}) recently proved that the semi-simple Reshetikhin-Turaev representations, in particular $\rho_{g,p}$, are complex variations of Hodge structures.  From the results of \cite{E6} it follows that $sh_{\rho_{g,p}}$ is  the Stein factorisation of a Griffiths' period mapping, whose monodromy is integral (see \cite{GM}). Theorem \ref{theointro1} suggests they may satisfy an infinitesimal Torelli theorem along each stratum.
\end{rem} 

Let $Y$ be a compact K\"ahler manifold, $\rho:\pi_1(Y)\to GL(N,\C)$  a semisimple
representation and $\theta$ its Higgs field (see \cite{Cor,Sim}). Then the 1-form $\omega_{\rho}={\rm tr}(\theta\wedge \theta^*)$ is semi-positive. We have the following 
result which appears in (\cite{E6}, see section 3.3.1 and Prop 3.3.1): 

\begin{lem}\label{oneform}
For every holomorphic map $f:Z\to Y$, where $Z$ is a complete K\"ahler manifold and $Y$ a 
compact K\"ahler manifold, the condition $f^*\omega_{\rho}\neq 0$ implies that $\rho(f_*(\pi_1(Z))$ is infinite. 
\end{lem}

The preimage of a one dimension stratum $\Gamma$ in $X_p$ is  a finite disjoint  union of  smooth curves $C^a_p(\Gamma) \subset X_p$. As such it defines a family of stable curves $C^a_p(\Gamma) \to \Mgnbar$. 
The general fiber is either: 
\begin{enumerate}
 \item  A stable curve all of whose components but one are rational curves with $3$ nodes or punctures and the remaining one is a rational curve with $l$ nodes and $k$ punctures with $2l+k=4$.
 \item  A stable curve all of whose components but one are rational curves with $3$ nodes or punctures and the remaining one is an elliptic curve with $1$  node (except if $g=1, n=1$ in which case $\overline{M^1_1}$ is its unique one-dimensional stratum).
\end{enumerate}

The second case is called a family of elliptic tails. 

As we will see, fusion rules in TQFT imply that $\rho_{g,p}$ restricted to 
the fundamental group of a stratum is essentially $\rho_{g,p}$ 
of the corresponding punctured surface describing that stratum. 
The $F$-conjecture states that the cone generated by the algebraic classes in $H_2(\overline{M_{g}^n};\R)$ is a polyhedral cone generated by the one-dimensional strata. 
If the F-conjecture holds, Lemma \ref{fonctoshaf} along with Lemma \ref{oneform} would reduce our claim to the critical cases $(g,n)=(0,4), (1,1)$.

\subsection{Fusion rules} 

Let  $\Mod(S)$ be the mapping class group of the compact 
connected surface $S$, possibly with boundary components and punctures. 
If $\Sigma$ has several connected components 
$S_1,\ldots,S_k$ then  $\Mod(S)$ states for the direct product 
$\prod_{i=1}^k\Mod(S_i)$. Furthermore we denote by 
$\PMod(S)$  the {\em pure} mapping class group consisting of the isotopy classes which fix {\em pointwise} the punctures and the boundary components. 

We denote by $\Sigma_{g,b}^n$ the 
genus $g$ orientable surface with $n$ punctures (or marked points) and $b$ boundary components.  We follow the standard convention to omit $b$ or $n$ when they are equal to 0. 

Let $S\subset \Sigma$ be a subsurface that need not be connected  but has finitely many components and no marked points on the boundary components. Then  we have a natural morphism of groups:

\[ \iota_S\buildrel{{\rm not.}}\over{=}\iota_{S\subset \Sigma}: \PMod(S) \to \PMod(\Sigma) 
\]
which associates to a mapping class on $S$ its extension to $\Sigma$ by the 
identity. 
 
 The Reshetikhin-Turaev representation $\bar \rho_{g+b}: \PMod(\Sigma_{g+b}) \to PU $ has a projective ambiguity that can be removed by passing to a central extension:
 \[ 1\to \Z \to \widetilde{\PMod(\Sigma_{g+b})} \to \PMod(\Sigma_{g+b}) \to 1
 \]
 whose class in $H^2(\PMod(\Sigma_{g+p}), \Z)$ is 
$12$ times the generator of $H^2(\PMod(\Sigma_{g+p}), \Z)\cong \Z$, when 
$g+p\geq 4$ (see \cite{MR}). Thus the class in  $H^2(\PMod(\Sigma_{g+p}), \Q)\cong H^2(M_{g+p}, \Q)$ is $12\lambda$, where $\lambda$ is the first Chern class of the Hodge bundle. 
The later also holds for $g+p=3$, although  $H^2(\PMod(\Sigma_{3}), \Z)\cong \Z\oplus \Z/2\Z$.

We may lift this extension under the natural map $\iota_{\Sigma_{g,b} \subset \Sigma_{g+b}}: \PMod(\Sigma_{g,b}) \to \PMod(\Sigma_{g+b})$ obtained by gluing a $\Sigma_{1,1}$ along each boundary component. This kills the
projective ambiguity of the Reshetikhin-Turaev representation of $\Sigma_{g,b}$, i.e. we obtain a central extension 

\[ 1\to \Z \to \widetilde{\PMod(\Sigma_{g,b})} \to \PMod(\Sigma_{g,b}) \to 1
\]
and a linear representation $\widetilde{\rho_{g,b,p, , (\mathbf{i})}}: \widetilde{\PMod(\Sigma_{g,b})} \to U(W_{g,p,(\mathbf{i})})$ 
where $(\mathbf{i})$ is a coloring of the boundary components and $W_{g,p,(\mathbf{i})}$ is the space of conformal blocks. 
 
If $S\subset \Sigma^n_{g,b}$ is connected we denote by $\widetilde{\PMod(S)}$ the pullback of the extension along $\iota_S$ and, in the non connected case, we 
 denote by $\widetilde{\PMod(S)}$ the product $\displaystyle{\prod_{k\in\pi_0(S)}} \widetilde{\PMod(S_k)}$. 
 
 This gives natural morphisms 
 
 \[ \widetilde{\iota_S}: \widetilde{\PMod(S)} \to \widetilde{\PMod(\Sigma_{g,b})}. 
\]

One has a natural central extension:

\[
1 \to \Z^n \to \PMod(\Sigma_{g,n}) \to \PMod(\Sigma_g^n) \to 1
\]
and the projective ambiguity of $\rho_{g,p, (\mathbf{i})}$ is resolved by $\widetilde{\rho_{g,p, (\mathbf{i})}}$. 
Then on the natural basis of  $\ker(\widetilde{\PMod(\Sigma_{g,n})} \to \PMod(\Sigma_g^n) ) \simeq \Z^{n+1}$ the representation $\rho_{g,p, (\mathbf{i})}$ acts 
by scalars which are $2p$-th roots of unity and
 $ \rho_{g,p,(\mathbf{i})}$ has a finite image if and only if  $ \widetilde{\rho_{g,p,(\mathbf{i})}}$ does. 
 
The most remarkable property of the Reshetikhin-Turaev representation is the following statement, see \cite{BHMV}: 

\begin{prop}\label{fus} If $S$ is obtained by cutting disjoint annuli surrounding a disjoint collection of $m$ pairwise non isotopic essential simple curves $\underline{c}$

\[ \widetilde{\iota_S}^* \widetilde{\rho_{g,p,(\mathbf{i})}} =\bigoplus_{\mathbf{j}\in \mathcal{C}_p^m} \widetilde{\rho_{S,(\mathbf{i}\amalg \mathbf{j} )}}
\]
 where $\mathbf{i}\amalg \mathbf{j}$ is the coloring of $S$ obtained by keeping $\mathbf{i}$ and coloring the $2m$ boundary components of $S$  
 in such a way that the two boundary components corresponding to the curve $c$  receive the color $j_{c}$. 
\end{prop}

\subsection{Analysis of the boundary}
Denote also by $\mathcal M(S)$ the moduli stack of curves homeomorphic to 
$S$ and $\overline{\mathcal M(S)}$ its Deligne-Mumford compactification. 
When $S$ is not connected $\mathcal M(S)$ and $\overline{\mathcal M(S)}$ 
is the product of the corresponding stacks associated to the connected components of $S$. 

Let $\underline{c}$ be an essential multicurve on $\Sigma_g^n$. 
Strata of $\Mgnbar$ are indexed by the topological type of the corresponding stable curves, or equivalently by orbits of the multicurves $\underline{c}$
under $\Mod_g^n$. 
If $S$ is a Riemann surface we denote by 
$S_{\underline{c}}$ the surface obtained by pinching curves in $\underline{c}$ 
to points which will be considered as marked points (or punctures) on $S_{\underline{c}}$. We also denote by $S\setminus \underline{c}$ the open surface obtained by cutting $S$ along the curves in $\underline{c}$, which can be identified with a punctured compact surface.  
  
Let $\Delta_{\underline{c}} \to \Mgnbar$ be the codimension $e$ stratum of the Deligne-Mumford stratification associated to the stable curve $S_{\underline{c}}$. This map is an immersion of a locally closed smooth substack. There exists a disjoint collection of 
pairwise non isotopic and disjoint essential simple curves in $\Sigma_g^n$ so that $\Delta_{\underline{c}}$ has an \'etale covering by 
$\mathcal{M}(S\setminus \underline{c})$ where $S$ is as in Proposition \ref{fus}. 

Let us assume for notation simplicity  $p$ is odd, the modifications being immediate  if $p$ is even. 

The preimage of the stratum $\Delta_{\underline{c}}$ in $X_p$ is a locally closed submanifold $i_p(\Delta_{\underline{c}}): \Delta_{\underline{c}}(X_p) \to X_p$  
which is \'etale-covered by a stack (actually a quasi projective manifold) equivalent to  $\Delta_{\underline{c}}\times_{\Mgnbar[p]} X_p$. 

If we divide by ${\rm Gal}(X_p/\Mgnbar[p])$  we get a substack 
$i(\Delta_{\underline{c}}[p]): \Delta_{\underline{c}}[p]\to \Mgnbar[p]$ which is actually $\Delta_{\underline{c}}[p]= \Delta_{\underline{c}} \times_{\Mgnbar[p]} \Mgnbar$ and $\Delta_{\underline{c}}(X_p)$ is a finite uniformization of $\Delta_{\underline{c}}[p]$.

\begin{lem}\label{stratum}
The smooth Deligne-Mumford  stack $\Delta_{\underline{c}}[p]$ is an \'etale gerbe over $\Delta_{\underline{c}}$ banded by an \'etale sheaf of  abelian groups which is locally equivalent to the constant sheaf with value $(\Z\slash p\Z) ^{e}$.
 
\end{lem}
\begin{proof}
 This follows from the local description of $\Mgnbar[p]$ given in \cite{mcgorb}. 
\end{proof}

\begin{lemma}
There exists a finite index subgroup  $H<\pi_1(\Delta_{\underline{c}}[p])$ which is identified to a quotient of $\PMod(S\setminus\underline{c})$ by a  subgroup contained in   $\ker(\iota_S^*\bar \rho_{g,p, (\mathbf{i})})$. 
 This identification carries $i(\Delta_{\underline{c}}[p])^*\rho_{g,p,(\mathbf{i})}|_H$ to $\iota_{S\setminus \underline{c}}^*\bar \rho_{g,p, (\mathbf{i})}$.
\end{lemma}
 \begin{proof}
The image of the homomorphism 
$f:\pi_1(\mathcal M(S\setminus \underline{c}))\to \pi_1(\Delta_{\underline{c}})$ has finite index in $\pi_1(\Delta_{\underline{c}})$. Define $H$ to be the 
preimage of $f(\PMod(S\setminus\underline{c}))$  
within $\pi_1(\Delta_{\underline{c}}[p])$. 
The inclusion map $i(\Delta_{\underline{c}}[p])$ induces a homomorphism   $\pi_1(\Delta_{\underline{c}}[p])\to \pi_1(\Mgnbar[p])$, which is 
covered by the homomorphism 
$\PMod(S\setminus\underline{c})/\PMod(S\setminus\underline{c})[p]\to
\Mod_g^n/\Mod_g^n[p]$ induced by $\iota_{S\setminus\underline{c}}$. 
Then the fusion rules from Proposition \ref{fus} prove the claim.  
\end{proof}

\begin{lem} There is a finite central extension of $H$ where the projective ambiguity of $\eta=i(\Delta_{\underline{c}}[p])^*\rho_{g,p,(\mathbf{i})}|_H$ is resolved by a representation $\widetilde {\eta}$, 
which identifies to  a quotient of $\widetilde{\PMod(S\setminus \underline{c})}$. The identification carries $\widetilde {\eta}$ to  $\iota_{S\setminus \underline{c}} ^* \widetilde{\rho_{g,p,(\mathbf{i})}}$. 
\end{lem}
\begin{proof}
Clear. 
\end{proof}

\begin{coro} \label{contcurvstrat} Assume $C \to \Delta_{\underline{c}}[p]$ is a curve which is mapped to a point by $sh_{\rho_{g,p}}$.  Then there exists 
a connected component $S'$ of $S\setminus \underline{c}$ 
and a curve $D \to \mathcal{M}(S')$ such that $D$ is mapped 
to a point by $sh_{\rho_{g(k),p}}$.
 
\end{coro}

 \subsection{Conclusion of the proof}
 
 Thanks to Proposition \ref{gibney}, Theorem \ref{theointro1} follows by induction from Lemma \ref{critstein} and the subsequent remark, from Corollary \ref{contcurvstrat} and from Proposition \ref{infiniteness}.
 
\subsection{A contraction}
In case $p=5$, families of  elliptic tails are contracted by $sh_{M_5}$, in particular it is a non trivial birational contraction.
If $n=0$ the morphism  $s_5: \overline{M_g} \to S$ given by the descent of $sh_{M_5}$ contracts precisely the divisor $\delta_1$  and factors through the divisorial contraction given by the big semiample class $11\lambda-\delta$ 
constructed by  \cite{CH}. 

In particular our statement is optimal for $p$ odd.  It is easy to see $\Mgnbar[5]$ satisfies the Shafarevich conjecture if $2g-2+n>0$.

\section{Proofs of Corollaries \ref{origami} and \ref{corointrod2}}\label{appteichcurves}

\subsection{Flat surfaces}
We seek for applications of previous results to the affine diffeomorphism group of a Veech surface. 
A punctured surface endowed with a geometric structure modeled on the complex plane and the group of its translations along with the $-1_{\C}$ is usually called a {\em flat} (or {\em half-translation}) surface if around the punctures the chart maps are of the form $z\to z^k$, $k\in \N^*$. It corresponds to the space obtained by identifying pairwise the edges of a collection of planar polygons by means of translations and $-1_{\C}$, where vertices give raise to punctures. A flat surface has an induced Riemann surface structure $X$ and the squared differentials of chart maps glue together to a well-defined holomorphic quadratic differential $q$, whose zero-set is the set of punctures. As it is well-known, conversely $(X,q)$ also defines a translation surface by means of the polygon associated to the periods of $\sqrt{q}$. Two flat surfaces are the same if there exists a holomorphic diffeomorphism between them which identifies the quadratic differentials. 
If the flat surface $(X,q)$ is obtained by 
the identification of pairs of parallel edges in a planar polygon $P$, 
and $h\in SL(2,\R)$, then define the translation surface $h\cdot (X,q)$
as the result of the same identification in the polygon $h(P)$. 
This is a well-defined action of $SL(2,\R)$ on the space $\mathcal H_g$
of translation surfaces of genus $g$.  When $q=\omega^2$, where 
$\omega$ is a holomorphic 1-form we retrieve a {\em translation} surface, 
namely one for which we only need translations and one can dispose of 
$-1_{\C}$. 

\subsection{Veech surfaces}
The Veech group $PSL(X,q)$ of $(X,q)$ is the image in $PSL(2,\R)$ of its  stabilizer $SL(X,q)$ with respect to the $SL(2,\R)$ action on $\mathcal H_g$. It is well-known that Veech groups are discrete subgroups in $PSL(2,\R)$ which are not cocompact.  
A flat surface is a {\em Veech surface} if its Veech group is a lattice in $PSL(2,\R)$.

Branched coverings of the flat torus with ramification above a single point 
are square-tiled surfaces, which provide the simplest 
examples of translation surface, also called {\em origamis}. 
As their Veech groups are commensurable with $PSL(2,\Z)$ all origamis are 
Veech surfaces. 
More generally, the $SL(2,\R)$ orbit of $(X,q)$ 
contains a square-tiled surface if and only if $SL(X,q)$ is 
$SL(2,\R)$-conjugate to a finite index subgroup of $PSL(2,\Z)$, by a result of Gutkin-Judge (\cite{GJ}).  

\subsection{Teichm\"uller curves}

The map sending $SL(2,\R)$ into the orbit of $(X,q)$ in $\mathcal H_g$ induces a natural map: 
\[ \phi_{(X,q)}:\mathbb H\to \mathcal T_g^n\]
where $\mathbb H$ is the upper half-plane and the number of punctures 
$n$ is the number of zeroes of $q$. As it is now well-known, 
see e.g. \cite{Ve}, 
$\phi_{(X,q)}$ is injective, isometric with respect to the 
hyperbolic metric on $\mathbb H$ and the Teichm\"uller metric on 
$\mathcal T_g^n$ and also holomorphic with respect to the 
natural complex structure on $\mathcal T_g^n$. Its image is called 
a Teichm\"uller disk centered at $(X,q)$. 
Then $\phi_{(X,q)}$ induces a holomorphic map:
\[ \Phi_{(X,q)}:\mathbb H/PSL(X,q)\to \mathcal M_g^n\]
called a Teichm\"uller curve. 
If $(X,q)$ is a Veech surface, then image of $\Phi_{(X,q)}$ is an algebraic curve,  whose lift to the moduli space of abelian differentials is affine for the natural affine structure,  which is normalized 
by $\mathbb H/PSL(X,q)$. In this case $\phi_{(X,q)}$ is 
proper and generically injective.

Let ${\rm Aff}^+(X,q)$ denote the group of orientation preserving affine diffeomorphisms of $(X,q)$, which coincides with the 
stabilizer of the Teichm\"uller disk in $\mathcal T_g^n$ 
with respect to the natural $\Mod_g^n$-action. The derivative map $D$ of an affine diffeomorphism provides us a surjective homomorphism 
encoded in an exact sequence: 

\[ 1\to {\rm Aut}(X,q)\to {\rm Aff}^+(X,q)\stackrel{D}{\to} SL(X,q)\to 1.\]
The kernel ${\rm Aut}(X,q)$ consists of the pointwise stabilizer of the Teichm\"uller disk. It follows that  ${\rm Aut}(X,q)$  is finite, when $g>1$.

\begin{proposition}\label{Veech}
Let $(X,q)$ be a Veech surface. Then $\rho_{g,p}({\rm Aff}^+(X,q))$ is infinite 
for odd $p\geq 7$ and even  $p\geq 14, p\not\in\{20,24\}$. 
\end{proposition}
\begin{proof}
The image of the fundamental group of the Veech surface under the 
algebraic map $\Phi_{(X,q)}$ which sends it into the Teichm\"uller curve 
is the group of affine diffeomorphisms of $(X,q)$. 
Then Theorem \ref{corointro1} implies the claim.  
\end{proof}

\begin{proof}[Proof of Corollary \ref{origami}]
Let $(X,q)$ be the origami Veech surface  associated to the 
ramified covering $\Sigma_g\to \Sigma_1$ of the torus $\Sigma_1$ branched over a single point $p\in \Sigma_1$. 
The corresponding unramified covering over  $\Sigma_1^*=\Sigma_1-\{p\}$ is  
characterized by a finite index subgroup $H\triangleleft \pi_1(\Sigma_1^*)\cong\mathbb F_2$. 
The Veech group $SL(X,q)$ of the origami  was described by 
Schmith\"usen (\cite{Sch}) as the image of the stabilizer of $H$ in ${\rm Aut}^+(\pi_1(\Sigma_1^*))$ into ${\rm Out}^+(\pi_1(\Sigma_1^*))\cong SL(2,\Z)$. 
It follows that ${\rm Aff}^+(X,q)$ coincides with the subgroup of mapping classes of 
those homeomorphisms of $\Sigma_g$ which lift homeomorphisms of $\Sigma_1$, namely the group $G$ from the statement. Then Proposition \ref{Veech} permits to conclude.    
\end{proof}

\subsection{Thurston's construction and proof of Corollary \ref{corointrod2}}
Consider multicurves $\underline{c}$ and $\underline{d}$ in 
minimal position on $\Sigma_g^n$ such that 
$\underline{c}\cup\underline{d}$ is connected. They are filling a subsurface $S$ if they are contained in $S$ and the complement  $S\setminus (\underline{c}\cup \underline{d})$ consists of a simply connected 
polygonal regions with at least 4 sides. Let us index components of 
$\underline{c}=(\gamma_i)_{1\leq i\leq m}$ and $\underline{d}=(\gamma_j)_{m+1\leq j\leq m+k}$. 

We will outline below a classical construction of Thurston (\cite[section 6]{Th}, see also \cite{Ve}) of a flat surface $(X,q)$  such that the corresponding Veech group contains $G(\underline{c},\underline{d})$.

If curves in $\underline{c}$ and $\underline{d}$ have parallel copies, let us assume they are associated the multiplicities $d_i\in\N$, for $1\leq i\leq m+k$.  Let $N$ be the 
associated geometric intersection matrix 
$N=(d_ii(\gamma_i,\gamma_j)_{1\leq i,j\leq m+k}$, where $N_{ii}=0$. 
Then $N$ is a Perron-Frobenius matrix and there exists an unique 
positive unit eigenvector $v$ such that $Nv=\mu v$, for some positive $\mu$.  

Consider the rectangles $R_p=[0,v_i]\times [0,v_j]\subset \C$ 
for every intersection point $p\in \gamma_i\cap \gamma_j$. 
We glue together $R_p$  to $R_q$ along the vertical or horizontal side whenever $p$ and $q$ are joined by an edge in $\underline{c}$ and $\underline{d}$, respectively, in the graph $\underline{c}\cup\underline{d}$.  
The differentials $dz^2$ on each rectangle glue together to a well-defined quadratic differential on the resulting surface $X$. 

Note that $q$ has a square root if and only if we can orient the curves $\gamma_i$ such that their geometric and algebraic intersection numbers 
coincide.

Define the multitwists 
$T_{\underline{c}}=\prod_{i=1}^m T_{\gamma_i}^{d_i}$ and 
$T_{\underline{d}}=\prod_{i=m+1}^{m+k} T_{\gamma_i}^{d_i}$. 
Then, Thurston proved in \cite[section 6]{Th} 
that $\langle T_{\underline{c}}, T_{\underline{d}}\rangle$ is a subgroup of ${\rm Aff}^+(X,q)$. Moreover, 
\[DT_{\underline{c}}=\left(\begin{array}{cc}
1 & \mu \\
0 & 1\\
\end{array}
\right), \; \; DT_{\underline{d}}=\left(\begin{array}{cc}
1 & 0 \\
-\mu & 1\\
\end{array}
\right).
\]
Furthermore $D\phi$ is elliptic, parabolic or Anosov if and only if 
$\phi$ is finite order, reducible (and actually a root of a positive 
multitwist) or pseudo-Anosov, respectively.

By a theorem of Bers and Kra (see \cite[Thm.6]{Bers} and 
\cite[Thm.1]{Kra}) for any pseudo-Anosov mapping class $\phi$ 
there exists an unique Teichm\"uller disk which is stabilized by $\phi$. 
If $(X,q)$ is any flat surface, then the multicurve systems for two different decompositions in cylinders and their moduli define by the construction above an affine equivalent flat surface. Thus every 
Teichm\"uller curve is associated to a pair of multicurves $\underline{c}$ and $\underline{d}$ along with multiplicities $d_i$. 
Conversely if $\mu\leq 2$ then $(X,q)$ generates a Teichm\"uler curve, namely its Veech group is a lattice. 
  
Nevertheless, it might happen that the subgroup generate by 
the Dehn twists $\langle T_{\underline{c}}, T_{\underline{d}}\rangle$ 
be of infinite index in the Veech group $SL(X,q)$.  Specifically Leininger showed in \cite[Thm. 7.1]{Lein} that: 

\begin{theorem}\label{leininger}
Let $\underline{c},\underline{d}$ be multicurves filling the surface $S$. 
The group $G(\underline{c},\underline{d})$  has finite index within the Veech group of the flat surface $(X,q)$ if and only if its configuration graph $\mathcal G(\underline{c}\cup\underline{d})$ is critical or recessive. 
\end{theorem}

\begin{proof}[Proof of Corollary \ref{corointrod2}]
Consider the subsurface $S'$ of genus $g'$ which is filled by the curves in $\underline{c}\cup\underline{d}$. We shall prove that the restriction $\rho_{g,p}$ to the image of $\PMod(S')$ within $\PMod(S)$ has infinite image. It is enough to show that $\rho_{g',p}(G(\underline{c},\underline{d}))$ is infinite, by applying the fusion rules.  
But this follows from proposition \ref{Veech} and Leininger's theorem 
\ref{leininger} above.   
\end{proof}

\section{Proof of Theorem \ref{H2}}
\subsection{Relative cohomological classifying maps}
For a connected paracompact space $X$, the universal covering $\widetilde X\to X$ is a 
principal $\pi_1(X)$-bundle whose isomorphism class defines a classifying map 
$f_X:X\to B\pi_1(X)$, which is unique up to homotopy. 
Then $f_X$ induces well-defined homomorphisms 
$f_X:H_*(X)\to H_*(\pi_1(X))$ and their duals 
$f_X^*: H^*(\pi_1(X))\to H^*(X)$.

\begin{lemma}
The homomorphisms $f_X:H_2(X)\to H_2(\pi_1(X))$  and 
 $f_X^*:H^2(\pi_1(X))\to H^2(X)$ 
are natural.
\end{lemma}
\begin{proof}
The Hurewicz homomorphisms
$h_X:\pi_*(X)\to H_*(X)$ are natural, i.e. a continuous map $\phi: Y\to X$ provides a commutative diagram:  
\[
\begin{array}{ccc}
\pi_2(Y) & \stackrel{\phi_*}{\to} & \pi_2(X)\\
h_X\downarrow & & \downarrow h_Y\\ 
H_2(Y) & \stackrel{\phi_*}{\to} & H_2(X)\\
\end{array}
\]
Now, a classical result of Hopf states that 
there exists an exact sequence: 
\[\pi_2(X)\stackrel{h_X}{\to} H_2(X)\stackrel{f_X}{\to} H_2(\pi_1(X))\to 0\]
This implies that $f_X$ are natural: 
\[
\begin{array}{ccc}
H_2(Y) & \stackrel{\phi_*}{\to} & H_2(X)\\
f_X\downarrow & & \downarrow f_Y\\ 
H_2(\pi_1(Y)) & \stackrel{\phi_*}{\to} & H_2(\pi_1(X))\\
\end{array}
\]
Dually $f_X^*$ are also natural.  
\end{proof}

Our next goal is to define the analog of maps $f_X$ in a relative context. 
Let now $X$ be a manifold of dimension at least 3 and $D\subset X$ be a polyhedron  of codimension at least 2. 
We denote its complement by $Y=X-D$. The inclusion $\phi: Y\to X$ induces therefore
an epimorphism $\phi: \pi_1(Y)\to \pi_1(X)$ whose kernel we denote by $K$. 
We have the following diagram with exact rows: 

\[
\begin{array}{rcccccccc}
0 \to H^1(\pi_1(X))& \to& H^1(\pi_1(Y)) &\to &H^1(K)^{\pi_1(Y)} &\stackrel{\delta}{\to} & H^2(\pi_1(X))& \stackrel{\phi^*}{\to}& H^2(\pi_1(Y))\\
\downarrow f_X^* & & \downarrow f_Y^*&  & & &f_X^*\downarrow & & \downarrow f_Y^*\\ 
H^1(X,Y) \to  H^1(X)& \to & H^1(Y)& \stackrel{\partial}{\to}& H^2(X,Y)&\stackrel{i}{\to} &H^2(X)&\stackrel{\phi^*}{\to}& H^2(Y)\\
\end{array}
\]
The top row is the five term exact sequence in cohomology associated to the exact sequence 
\[ 1\to K\to \pi_1(Y)\to \pi_1(X)\to 1\]
and the bottom row is the long exact sequence in cohomology associated to the 
pair $(X,Y)$. 

\begin{lemma}
Assume that $i$ is injective. Then there is a natural map 
\[ f_{X,Y}^*: H^1(K)^{\pi_1(Y)}\to H^2(X,Y)\]
so that all squares in the diagram above are commutative.  
In particular this is so if $H^1(Y)=0$. 
\end{lemma}
\begin{proof}
If $x\in H^1(K)^{\pi_1(Y)}$ we define 
\[ f_{X,Y}^*(x)=i^{-1}(f_X^*(\delta(x)))\]
This is a well-defined homomorphism making the diagram above commutative if 
$f_X^*(\delta(x)))$ belongs to $i(H^2(X,Y))$. 
As the top row is exact we have: 
\[ \delta(H^1(K)^{\pi_1(Y)})=\ker(H^2(\pi_1(X)) \stackrel{\phi^*}{\to} H^2(\pi_1(Y)))\]
The functoriality of $f_X^*$ implies that 
\[ \phi^*(f_X^*(\delta(H^1(K)^{\pi_1(Y)})))=0\]
and the exactness of the bottom row implies that 
\[ f_X^*(\delta(H^1(K)^{\pi_1(Y)}))\subseteq i(H^2(X,Y))\]
as needed. 
\end{proof}

\subsection{Hochschild-Serre exact sequence}
We now apply the previous construction to the moduli stack of curves. 

Observe that $N_{g,n}$ is also the minimal number of normal generators 
of $\Mod_g^n[p]$ within $\Mod_g^n$. 
We now follow the idea used in \cite{FP2} to analyse the 
image of the quantum representation. Specifically, we have a general upper bound:

\begin{lemma}\label{bound}
Let $\{a_1,a_2,\ldots,a_{N_{g,n}}\}$ be a minimal system of 
normal generators for $\Mod_g^n[p]$ within 
$\Mod_g^n$. 
Then the evaluation 
homomorphism 
\[E:{\rm Hom}(\Mod_g^n[p], \R)^{\Mod_g^n}
\to \R^{N_{g,n}},\]
given by 
$ E(f)= (f(a_1),f(a_2),\ldots, f(a_{N_{g,n}}))$
is injective. 
\end{lemma}
\begin{proof}
Any element $x\in \Mod_g^n[p]$ is a product 
$x=\prod_i g_ia_ig_i^{-1}$, for some $g_i\in \Mod_g^n$.
Since $f\in {\rm Hom}(\Mod_g^n[p],\R)^{\Mod_g^n}$ is conjugacy 
invariant we have 
$ f(x)=\sum_{i}f(g_ia_ig_i^{-1})=\sum_if(a_i)$ and the Lemma follows.
\end{proof}

We can now compare the second cohomology to a module of invariants: 

\begin{proposition}\label{hom}
We have an exact sequence
\[ 0\to {\rm Hom}(\Mod_g^n[p], \R)^{\Mod_g^n}\stackrel{\partial}{\to} 
H^2(\Mod_g^n/\Mod_g^n[p], \R)\to 
H^2(\Mod_g^n,\R).\] 
In particular $\dim H^2(\Mod^n_g/\Mod^n_g[p],\R))\leq n+1+N_{g,n}$, if $g\geq 4$.
\end{proposition}
\begin{proof}
The  5-term exact sequence in cohomology associated to the exact sequence  
\[ 1\to \Mod_g^n[p] \to\Mod_g^n \to \Mod_g^n/\Mod_g^n[p]\to 1,\]
gives us:
\[ H^1(\Mod_g^n,\R)\to {\rm Hom}(\Mod_g^n[p], \R)^{\Mod_g^n}\stackrel{\partial}{\to} 
H^2(\Mod_g^n/\Mod_g^n[p], \R)\to 
H^2(\Mod_g^n,\R).\]
Now $H^1(\Mod_g^n,\R)=0$ and 
$H^2(\Mod_g^n,\R)=\R^{n+1}$. Therefore  
\[ \dim {\rm Hom}(\Mod_g^n[p], \R)^{\Mod_g^n}\leq 
\dim H^2(\Mod_g^n/\Mod_g^n[p],\R))\leq n+1+\dim {\rm Hom}(\Mod_g^n[p], \R)^{\Mod_g^n}. \]
Then lemma \ref{bound} shows that 
$\dim{\rm Hom}(\Mod_g^n[p], \R)^{\Mod_g^n}\leq n+1+N_{g,n}$ and Proposition \ref{hom} follows. 
\end{proof}

\subsection{Presentations for $\Mod_g^n[p]$} 
Before going further recall that the geometric intersection 
$i(\gamma,\gamma')\in \N$ of two isotopy classes of 
simple closed curves $\gamma$ and $\gamma'$ is the minimal number of intersection points between curves in the given classes.

We shall use now the following result of Dahmani from \cite{Da}:
\begin{theorem}
There exists some $p_0$ (depending on $g$ and $n$) such that if $p$ is divisible by $p_0$, then  
the group $\Mod_g^n[p]$ has an infinite presentation as follows: 
\begin{enumerate}
\item the generators are $x_{\gamma}$, where $\gamma$ belongs to a set  
$\Omega_p(\Sigma_g^n)$ of representatives for the isotopy classes 
of simple closed curve on the surface $\Sigma_g^n$ modulo $\Mod_g^n[p]$;  
\item the relations are  
\[ [x_{\gamma}, x_{\gamma'}]=1, \; {\rm if }\;  i(\gamma,\gamma')=0\] 
\end{enumerate}
\end{theorem}

\begin{corollary}
When $p$ is divisible by $p_0$ the vector space ${\rm Hom}({\Mod}_g^n[p], \R)^{\Mod_g^n}$ is isomorphic to $\R^{N_{g,n}}$, 
where $N_{g,n}$ is the number of 
orbits of $\Mod_g^n$ on the set of isotopy classes of simple closed curves on $\Sigma_g^n$. 
\end{corollary}
\begin{proof}
The free abelian group ${\rm Hom}({\Mod_g^n}[p], \R)$ is freely generated by 
$\Omega_p(\Sigma_g^n)$. Since the action of $\Mod_g^n$ on ${\rm Hom}({\Mod_g^n}[p], \R)$ is by conjugacy, the invariant module 
${\rm Hom}({\Mod}_g^n[p], \R)^{\Mod_g^n}$ is the free 
abelian group generated by $\Omega_p(\Sigma_g^n)/\Mod_g^n$. 
As $\Omega_p(\Sigma_g^n)/\Mod_g^n$ is finite of cardinal $N_g^n$, the 
claim follows. 
\end{proof}

This already implies that:
\begin{corollary}
If $g\geq 4$ and $p$ is divisible by $p_0$, then $H^2({\Mod_g^n/\Mod_g^n[p]},\R)\neq 0$. 
\end{corollary}


\vspace{0.5cm}

    

\begin{thebibliography}{99}

\bibitem{AMU}
J.~E.~Andersen, G.~Masbaum and K.~Ueno, {\em 
Topological quantum field theory and the Nielsen–Thurston classification of M(0,4)}, 
Math. Proc. Camb. Philos. Soc. {\bf 141} (2006), 477--488. 


\bibitem{AS}
J. Aramayona and J. Souto, {\em Rigidity phenomena in mapping class groups}, Handbook of Teichm\"uller Theory VI
(A.Papadopoulos, Ed.), IRMA Lect. Math. Theor. Phys. 27, European Math. Soc., Z\"urich, 2016, 131--165.

\bibitem{Ara}
D. Arapura, {\em Toward the structure of fibered fundamental groups of projective manifolds}, J.\'Ecole Polytechnique
4 (2017), 595--611. 




 
\bibitem{Bers}
L.~Bers, {\em 
An extremal problem for quasiconformal mappings and a theorem by Thurston},
Acta Math. {\bf 141} (1978),  73--98.




 


\bibitem{BHMV} 
C.~Blanchet,  N.~Habegger, G.~Masbaum, P.~Vogel, {\em  Topological quantum field theories derived from the Kauffman bracket},   Topology  {\bf 34} (1995),  883--927.


\bibitem{CCE} F.~Campana, B. ~Claudon, P.~Eyssidieux,  {\em Repr\'esentations lin\'eaires des groupes k\"ahl\'eriens : Factorisations et conjecture de Shafarevich lin\'eaire},    Compositio Math. {\bf 151} (2015),  351--376.


\bibitem{Ca}
F. Catanese, {\em Kodaira fibrations and beyond: methods for moduli theory}, Jpn. J. Math. 12 (2017), no.
2, 91--174.




\bibitem{Cor}
K. Corlette, {\em 
Flat G-bundles with canonical metrics}, 
J. Differ. Geom. 28 (1988), No. 3, 361--382.

\bibitem{CH} M.~Cornalba, J.~Harris {\em Divisor classes associated to families of stable varieties, with applications to the moduli space of curves },  Annales scientifiques de l'\'Ecole Normale Sup\'erieure {\bf 21} (1988),  455--475. 

\bibitem{Da}
F.~Dahmani, {\em 
The normal closure of big Dehn twists and plate spinning with rotating families}, Geom. Topol. {\bf 22} (2018), 4113--4144. 

\bibitem{DW}
G.D.Daskalopoulos,  R.A.Wentworth, 
{\em Harmonic maps and Teichmüller theory}, 
Handbook of Teichmüller theory I (A. Papadopoulos Ed.), IRMA Lect. Math. Theor. Phys. 11, European Math. Soc.,Zürich, 2007, 33--109.


\bibitem{DMu} P.~Deligne, D.~Mumford, {\em The irreducibility of the space of curves of a given genus}, Publ. Math. I.H.E.S. {\bf 36} (1969), 75--109.

\bibitem{E6}{P.~Eyssidieux }
{\em Sur la convexit\'e
holomorphe des rev\^etements lin\'eaires
r\'eductifs d'une vari\'et\'e projective alg\'ebrique complexe},
 Inventiones Math. {\bf 156} (2004), 503--564.
 
\bibitem{ajm} P.~Eyssidieux, {\em Orbifold K\"ahler Groups and the Shafarevich Conjecture for Hirzebruch's covering surfaces with equal weights}, Asian J. Math {\bf 22} (2018), special volume in the honor of N. Mok,  315--330.

\bibitem{mcgorb} 
P.~Eyssidieux,  L.~Funar, {\em Orbifold K\"ahler Groups related to Mapping Class groups},  	arXiv:2112.06726.


\bibitem{F}L.~Funar, {\em  On the TQFT representations of the mapping class groups},   Pacific J. Math. {\bf 188} (2)(1999), 251--274.  

\bibitem{FK}   
L.~Funar,  T.~Kohno, {\em On Burau representations at roots of unity},
Geom. Dedicata {\bf 169} (2014), 145--163.   

\bibitem{LF} L.~Funar,  P.~Lochak {\em Profinite Completion of Burnside type quotients of surface groups}  Comm. Math. Phys. {\bf 360} (2018), 1061--1082. 

\bibitem{FP2} L.~Funar,  W.~Pitsch, {\em  Images of quantum representations of mapping class groups and Dupont-Guichardet-Wigner quasi-homomorphisms},
J. Inst. Math. Jussieu {\bf 17} (2018), 277--304.

\bibitem{GKM} A.~Gibney, S.~Keel, I.~Morrison {\em Towards the ample cone of $\overline{M}_{g,n}$}, Journal A. M. S., {\bf 15} (2002), 273--294.

\bibitem{GM}
P.~Gilmer, G.~Masbaum, {\em  
Integral lattices in TQFT},   
Ann. Sci. Ecole Norm. Sup.  S\'erie (4)  {\bf 40} (2007),  815--844.

\bibitem{GM2}
P.~Gilmer, G.~Masbaum, {\em 
 Maslov index, Lagrangians, mapping class groups and TQFT},
Forum Math. {\bf 25} (2013),  1067--1106. 

\bibitem{Go1}
P. Godfard, {\em Rigidity of Fibonacci representations of mapping class groups}, 
Ann. Institut Fourier 75 (2025), 2529--2563.


\bibitem{Go3}
P. Godfard, {\em Hodge structures on conformal blocks}, arXiv:2406.07459. 

\bibitem{Go4}
P. Godfard, {\em 
Rigidity of $SU(2)$ and $SO(3)$ quantum representations of mapping class groups at prime levels, arXiv:2511.19795. }







\bibitem{GJ}
E.~Gutkin and C.~Judge, {\em 
Affine mappings of translation surfaces: Geometry and arithmetic},
Duke Math. J.  {\bf 103} (2000), 191--213. 

\bibitem{Knu} F.~Knudsen, {\em The projectivity of the moduli space of stable curves, II: The stacks $M_{g,n}$}, Math. Scand. {\bf 52} (1983),  161--199.

\bibitem{Knu2} F.~Knudsen, {\em The projectivity of the moduli space of stable curves, III: The  line bundles on $M_{g,n}$ and a proof of the projectivity of $\overline{M_{g,n}}$ in characteristic $0$}, 
Math. Scand. {\bf 52} (1983),  200--212. 

\bibitem{Kra}
I.~Kra, {\em  On the Nielsen-Thurston-Bers type of some self-maps of Riemann surfaces}, Acta Math. {\bf 146} (1981),  231--270. 

\bibitem{KoSa} T.~Koberda, R.~Santharoubane, {\em Quotients of surface groups and homology of finite covers via quantum representations}, Inventiones Math. {\bf 206} (2016), 
269--292. 

\bibitem{Kol} J.~Koll\'ar, {\em Projectivity of complete moduli}, Journal of Differential Geometry {\bf 32} (1990) 235-268.

\bibitem{Kor} J.~Korinman, {\em On the (in)finiteness of the image of Reshetikhin-Turaev representations}, Archiv Math. 111 (2018), 247--256. 

\bibitem{Ku}
G. Kuperberg, {\em Denseness and Zariski denseness
of Jones braid representations}, Geom. Topol. 15 (2011), 11--39. 



\bibitem{LPS} Y.~Laszlo, C.~Pauly,  C.~Sorger, {\em On the monodromy of the Hitchin connection}, J. Geom. Phys. {\bf 64} (2013), 64--78. 


\bibitem{Lein}
C.~J.~Leininger, {\em  On groups generated by two positive multi-twists: Teichmüller curves and Lehmer’s number}, 
Geom. Topol. {\bf 8} (2004),  1301--1359. 


\bibitem{Mas}
G.~Masbaum, {\em An Element of Infinite Order in TQFT-Representations of Mapping Class
Groups},  Low-Dimensional Topology (Funchal, 1998), 137--139, Contemp. Math. 
233,  A. M. S., 1999.

\bibitem{MR}
G.~Masbaum,  J.~Roberts, {\em 
On central extensions of mapping class groups},  
Math. Ann.  {\bf 302}  (1995),  131--150.

\bibitem{No1} B.~Noohi, {\em Fundamental groups of algebraic stacks},   
J. Inst. Math. Jussieu {\bf 3} (2004), 69--103.

\bibitem{No2} B.~Noohi, {\em Foundations of topological stacks}, arXiv math/0503247 (2005).
\bibitem{Sch}
G.~Schmith\"usen, {\em An algorithm to finding the Veech group of an origami}, 
Experim. Math. {\bf 13} (2004), 459--472. 

\bibitem{Shi}
H. Shiga, {\em On monodromies of holomorphic families of Riemann surfaces and modular
transformations}, Math. Proc. Cambridge Philos. Soc. 122 (1997), 541--549.

\bibitem{Sim1}
C.T. Simpson, {\em 
Constructing Variations of Hodge Structure using Yang-Mills
theory and applications to uniformization},  J. Am. Math. Soc. 1 (1988),  867--918. 

\bibitem{Sim}
C.T. Simpson, {\em Higgs bundles and local systems},  
Publ. Math. I. H. E. S. 75 (1992), 5--95.

\bibitem{Sim2}
C.T. Simpson, {\em Subspaces of moduli spaces of rank one local systems}, Annales scientifiques de l’É.N.S. 4e série, 26 (1993), 361--401. 

\bibitem{Th} 
W.~P.~Thurston, {\em On the geometry and dynamics of diffeomorphisms of surfaces}, Bull. A.M.S. {\bf 19} (1988), 417--431. 

\bibitem{Ve} W.~A.~Veech, {\em Teichm\"uller curves in moduli spaces, Eisenstein series and 
an application to triangular billiards}, Inventiones Math. {\bf 97} (1989), 553--583. 


  
   \end{thebibliography}
\end{document}